\documentclass[a4paper,12pt]{article}

\usepackage[english]{babel}
\usepackage[margin=1in]{geometry} 

\title{Spectral and computational aspects of a regularized fractional Laplacian for non-local diffusion on graphs}
\author{Alessandro Filippo\thanks{University of Rome Tor Vergata, Department of Mathematics, Italy (\texttt{filippo@mat.uniroma2.it}, \texttt{mariarosa.mazza@mat.uniroma2.it})}
\and Mariarosa Mazza\footnotemark[1]
}

\usepackage{amsmath}
\usepackage{amssymb}
\usepackage{amsthm}
\usepackage{amsfonts}  
\usepackage{bm}
\usepackage[thinc]{esdiff}
\newtheorem{theorem}{Theorem}
\newtheorem{lemma}{Lemma}
\newtheorem{definition}{Definition}
\newtheorem{remark}{Remark}
\newtheorem{example}{Example}
\newtheorem{proposition}{Proposition}

\usepackage{leftidx}

\usepackage{graphicx}

\usepackage[normalem]{ulem}

\usepackage{tkz-graph}

\usepackage{subcaption}

\usepackage{enumitem}

\usepackage{booktabs}
\usepackage{latexsym}
\usepackage{multirow}

\usepackage{hyperref}

\usepackage[dvipsnames]{xcolor}

\usepackage{amsopn}
\DeclareMathOperator{\diam}{diam}
\DeclareMathOperator{\diag}{diag}
\DeclareMathOperator{\dist}{dist}
\DeclareMathOperator{\restr}{restr}
\DeclareMathOperator*{\argmin}{\arg\!\min}

\newcommand{\R}{\mathbb{R}}
\newcommand{\cE}{\mathcal{E}}
\newcommand{\cV}{\mathcal{V}}
\newcommand{\cS}{\mathcal{S}}

\newcommand{\onevec}{\mathbf{1}}
\newcommand{\xvec}{\mathbf{x}}
\newcommand{\yvec}{\mathbf{y}}

\newcommand{\zerovec}{\mathbf{0}}
\newcommand{\consensus}{\overline{\mathbf{x}}_{\mathbf{0}}}
\newcommand{\fL}{\Delta^\alpha}
\newcommand{\rfL}{{}_r\fL}
\newcommand{\srfL}{{}_r^\beta\fL}

\newcommand{\fW}{W^\alpha}
\newcommand{\rfW}{{}_r\fW}
\newcommand{\srfW}{{}_r^\beta\fW}
\newcommand{\fG}{G^\alpha}
\newcommand{\rfG}{\leftidx{_r}\fG}
\newcommand{\srfG}{\leftidx{_r^\beta}\fG}

\begin{document}

\maketitle

\abstract{ 
The fractional Laplacian has been widely used to model non-local diffusion on graphs, allowing interactions that extend beyond immediate neighbors. However, it suffers from a structural inconsistency as it breaks compatibility with the topology of the original network. 
To address this issue, a combination of the standard and fractional Laplacians aimed at restoring compatibility while retaining the spectral richness of the fractional operator was recently proposed. 

In this work, we provide a thorough analysis of the diffusion properties of the resulting regularized operator. We prove that it yields superdiffusive behavior independently of whether the underlying graph is weighted or unweighted—a property not generally satisfied by other non-local Laplacian variants. An efficient construction of the regularized operator which preserves the same asymptotic computational cost of the fractional Laplacian is also given.
Our numerical experiments demonstrate the effectiveness and computational practicality of the regularized operator for modeling non-local diffusion on real-life networks.
}

\maketitle
	
\section{Introduction}

Modeling diffusion processes on networks has been a central topic of research, with applications including the spread of infectious diseases, opinion dynamics on social media, signal propagation in neural and brain networks, and the analysis of transport and communication infrastructures. Graph-based models are particularly well-suited for capturing the complexity of interactions in real-world networks, where the underlying topology strongly influences the system dynamics. A central tool in this context is the graph Laplacian, a local operator that encodes nearest-neighbor interactions and gives rise to the consensus dynamics described by the heat equation on graphs.

However, many natural and engineered systems exhibit dynamics that are not purely local: information, mass, or influence can jump across long distances, bypassing intermediate connections. These \emph{non-local} behaviors cannot be adequately captured by standard Laplacian-based diffusion models. To address this issue, two distinct approaches have emerged. 

The first approach builds upon the notion of the \emph{fractional graph Laplacian} defined algebraically via spectral calculus as a fractional power of the standard graph Laplacian, and it founds applications for instance in image reconstruction \cite{Buccini_fractional_laplacian_image_reconstruction} or graph neural networks \cite{maskey2023fractional}. It was first studied in \cite{riascos2012levy,Riascos_introduction_of_fractional_laplacian} where the authors introduced the notion of fractional random walks on undirected networks using the operator $\Delta^\alpha$, with $\Delta$ the graph Laplacian and $\alpha \in (0,1)$, and showed that this leads to Lévy-type superdiffusion; see also the recent survey on non-local dynamics on undirected weighted networks \cite{Riascos_survey_nonlocal_dynamics} for additional details on this class of models. Further developments were carried out in \cite{Benzi_2020_fractional_laplacian}, where the authors extended the notion of non-local dynamics to
directed networks and in \cite{Bertaccini_Dustantante_nonlocal_varaible_index} where a variable diffusion order was considered. Efficient numerical algorithms for the computation of the solution to the fractional diffusion equation on directed networks using rational Krylov methods were later proposed in \cite{benzi2022rational}.

In parallel, Estrada \cite{estrada2012path_laplacian} proposed a combinatorial approach to capturing non-local interactions on undirected and unweighted networks through the notion of \emph{path Laplacian}. This operators extend the standard Laplacian by incorporating the influence of shortest paths between nodes, rather than just direct connections. For an alternative approach based on walks see \cite{ArrDur2026_walkbasedlaplaciansmodeling}.

Despite their differing foundations, both non-local diffusion operators, fractional and path-based Laplacians, share a key feature: they replace the underlying graph with a complete weighted graph, with edge-weights that typically decay as a power-law of the combinatorial graph distance separating the two nodes. This allows information/mass to ``jump'' from a node to any other node in the network, even if they were not originally connected by an edge. As a consequence of this, one would expect that a non-local diffusion process on a network based on the fractional Laplacian or the path Laplacian operators always converges faster than the standard diffusion. 
However, as already noted by Estrada in \cite{Estrada_2021_fractional_vs_path_laplacians}, this holds true for the path Laplacian in presence of unweighted graphs, while it ultimately depends on the graph, in the case of the fractional Laplacian, as it could happen that the complete fractional graph induced by $\fL$ has a lower algebraic connectivity than the original graph  (see Section \ref{subsec:model_diffusion} for more details).

On the other hand, replacing the original graph with a fully connected one where weights decay with distance may introduce a structural inconsistency: the resulting dynamics could be 
no longer strictly compatible with the topology of the original graph. This incompatibility is highlighted and addressed in \cite{Bianchi2022}, where the authors proposes a \emph{regularized fractional graph Laplacian} constructed as a combination of the standard and fractional Laplacians that retains the spectral richness of fractional operators while enforcing compatibility with the underlying graph structure.

However, while numerical experiments demonstrated that the regularized operator can produce genuine non-local diffusion, the question of whether it systematically yields faster diffusion than the standard Laplacian remains open. Moreover, no practical method for computing the regularized Laplacian efficiently was available, leaving its use limited in large-scale settings.

In this work, we further this line of investigation by proving that the algebraic connectivity of the regularized fractional graph is always greater than that of the standard graph, and in some cases even greater than that of the fractional graph. This, in turn, will guarantee that the regularized fractional Laplacian ensures a superdiffusive behavior compared to the standard Laplacian, for any choice of the fractional parameter, and independently of whether the graph is weighted or unweighted —a property that does not hold in general for either the fractional or the path Laplacians. 
Moreover, we introduce a Boolean--Hadamard formulation for constructing the regularized Laplacian, which allows its computation at the same cost as that of the fractional Laplacian alone. 

The remaining is organized as follows. Section~\ref{sec:basics} contains some basic concepts of graph theory and linear algebra that are of interest to this work. Section~\ref{sec:fractional_laplacian} introduces the fractional graph Laplacian and its use to model non-local diffusion on graphs. Section~\ref{sec:reg_fractional_laplacian} investigates the spectral behavior of the regularized Laplacian and introduces a new Boolean–-Hadamard formulation for its computation. Section~\ref{sec:numerical_examples} presents a series of numerical experiments on both synthetic and real-world graphs comparing spectral and dynamical properties of standard, fractional, path-based,
and regularized Laplacians, illustrating the benefits and limitations of each model in practical scenarios. Finally, in Section~\ref{sec:conclusions} we draw conclusions.
\newline

\noindent\textbf{Notation.} In the following, $ \circ $ indicates the \emph{Hadamard} (or \emph{entrywise}) \emph{product} and $\mathbf{1}=[1,\dots,1]^T$ denotes the vector of all ones. Moreover, for each matrix $M\in\R^{n \times n}$ and each vector $\mathbf{m}\in \R^{n}$, we will denote by 
\begin{itemize}
    \item $\diag(M)$ the diagonal matrix such that $(\diag(M))_{ii}=(M)_{ii}$;
    \item $\diag(\mathbf{m})$ the diagonal matrix such that $(\diag(\mathbf{m}))_{ii}=m_{i}$. 
\end{itemize}
Notice that with the above notation we have that $M = \diag(M) = \diag(M\onevec)$ if and only if $M$ is a diagonal matrix. 

\section{Background} \label{sec:basics}

In this section, we recall some basic notions of graph theory (Section \ref{subsec:graphs}) and the algebra of Boolean matrices (Section \ref{subsec:lin_algebra}) that will be useful for this work. Moreover, we recall the definition of graph Laplacian as well as its role in modeling diffusion processes on graphs (Section \ref{subsec:graph_laplacian}). Additional material can be found in \cite{harary1969graph,kim_boolean_1982,Bullo_lectures}.

\subsection{Graphs} \label{subsec:graphs}

We define a \emph{graph} as a pair $G = (\cV, w)$ where $\cV=\{1,\dots,n\}$ is a set of \emph{nodes}, and $w:\cV \times \cV \longrightarrow \R_{\geq 0}$ is a function that assigns a non-negative weight $w_{ij}$ to each pair of nodes $(i, j)$. A pair $(i, j)$ such that $w_{ij} > 0$ is called \emph{edge} from $i$ to $j$, and we denote by $\cE$ the collection of all edges in $G$.

We say that a graph $G$ is \emph{undirected} if $w_{ij}=w_{ji}$ for all $i,j=1\dots,n$, otherwise we say that $G$ is \emph{directed}. If all weights are in $\{0,1\}$ we say that the graph is \emph{unweighted}, otherwise the graph is said to be \emph{weighted}. A graph is said to be \emph{loopless} if it does not contain any edge from a node to itself. A graph is said to be \emph{complete} if there is an edge connecting each pair of nodes $i\neq j$. 
 We say that a graph $G$ is \emph{sparse} if it has at most $O(n)$ edges, where $n$ is the number of nodes, and we say that a graph is \emph{very dense} if it has almost all possible edges; in other words, a very dense graph is an almost complete graph.

A \emph{walk} of length $\ell$ in $G$ is a sequence of $\ell +1 $ nodes $i_0,i_1,\dots,i_\ell$ such that $(i_k,i_{k+1})\in\cE$ for each $k=0,\dots,\ell-1$.  A walk without repeated nodes is called a \emph{path}, and the \emph{combinatorial distance} $\dist(i,j)$ between two nodes $i$ and $j$ is defined as the length of the \emph{shortest path} that connects them. The \emph{diameter} of a graph is defined as $\diam=\max_{ij}\{ \dist(i,j)\}.$
We say that an undirected graph $G$ is \emph{connected} if there exists a walk connecting each pair of nodes. A \emph{connected component} is a connected subgraph of $G$ that is maximal with respect to inclusion.

For each graph $G$ we introduce the \emph{weight matrix} $W\in \R^{n \times n}$ as the matrix $(W)_{ij} = w_{ij}$, and the \emph{adjacency matrix} $A$ as the $n \times n$ matrix with entries
\begin{equation*}
(A)_{ij} = a_{ij}= 
    \begin{cases}
        1 & \text{ if } (i,j)\in \cE, \\
        0 & \text{ otherwise. }
    \end{cases}
\end{equation*}
Notice that $W$ and $A$ are symmetric if and only if $G$ is undirected and that $W = A$ if and only if $G$ is unweighted. Moreover, if $G$ is loopless, then $W$ and $A$ are \emph{hollow}, i.e., $a_{ii}=w_{ii}=0$ for all $i=1,\dots,n$. Lastly, $G$ is sparse if and only if $W$ and $A$ are sparse, i.e., they have at most $O(n)$ nonzero entries.

Henceforth, unless specified otherwise, we will assume that each graph is undirected, connected and loopless, and then we use the word ``graph'' to indicate a graph with these properties. 

\subsection{Boolean matrices} \label{subsec:lin_algebra}

A \emph{Boolean} or \emph{binary matrix} is a matrix with entries in a Boolean algebra \cite[Section 1.2]{kim_boolean_1982}. For this work, we will consider the set of Boolean matrices as the set of matrices with entries in $\{0,1\}$ equipped with the standard matrix operations plus the definition of \emph{complementary matrix}.
\begin{definition}
Given a Boolean matrix $B\in \{0,1\}^{n \times n}$, its complementary matrix $B^{C}$ is such that $(B)_{ij}=1$ if and only if $(B^C)_{ij}=0$.  
\end{definition}

Three special Boolean matrices are worth mentioning: the \emph{universal matrix} $\mathbf{1}\mathbf{1}^T$, in which each entry is $1$, the \emph{identity matrix} $I$, and the \emph{zero matrix} $O$. Notice that all these matrices have some significant properties regarding the Hadamard product and the complement of a Boolean matrix: 
\begin{itemize}
    \item for all $B\in \{0,1\}^{n \times n}$, $B + B^C = \mathbf{1}\mathbf{1}^T$;
    \item for all $M \in \R^{n \times n}$, $M \circ \mathbf{1}\mathbf{1}^T = M$, $M\circ I = \diag(M)$, and $M\circ O = O$.
\end{itemize}

In our main results, we will make use of the following observations that highlight some connections between graph theory and the algebra of Boolean matrices. 

\begin{remark}
    If $A$ is the adjacency matrix of a graph $G$ with $n$ nodes, then $A$ is a  $n\times n$ Boolean matrix. Moreover, $A$ represents the collection of its edges $\cE$ as a binary relation over the set $\{1,\dots,n\}$ \cite[Section 1.2]{kim_boolean_1982}. Also, if $W$ is the weight matrix of $G$, then $W\circ A = W$.
\end{remark}

\begin{remark} \label{rmk:dense_graphs}
Let $A$ be the adjacency matrix of a graph $G$, then $G$ is sparse if and only if $A^C$ has almost $n^2$ nonzero entries. 
   
\end{remark}

\subsection{Graph Laplacian and local diffusion on graphs} \label{subsec:graph_laplacian}
Given a graph with weight matrix $W$, we define the \emph{graph Laplacian} as
\begin{equation*} 
    \Delta = D - W,
\end{equation*}
where $D$ is the diagonal \emph{degree matrix} such that $D=\diag(W\onevec)$. Entrywise, we have that
\begin{equation*} 
(\Delta)_{ij}=
    \begin{cases}
        \sum_{k=1}^n w_{ik} & \text{ if $i= j$} \\
    -w_{ij}    & \text{ if $i\neq j$}. 
    \end{cases}
\end{equation*}

\begin{remark}
    If $\Delta$ is the Laplacian matrix of a graph with weight matrix $W$, then $\Delta\circ B = - W\circ B$ for each Boolean and hollow matrix $B$. 
\end{remark}

Notice that $\Delta\mathbf{1}=\mathbf{0}$ for any graph $G$, and that $\Delta$ is symmetric and semipositive definite if and only if $G$ is undirected. 
Moreover, it is well known that the multiplicity of $0$ as an eigenvalue of $\Delta$ is equal to the number of connected components of $G$. In particular, $0$ is a simple eigenvalue of $\Delta$ if and only if $G$ is connected. The smallest positive eigenvalue of $\Delta$ is often called \emph{algebraic connectivity} of the graph, and its associated eigenvector is also referred to as \emph{Fiedler eigenvector} \cite{Fiedler_algebraic_connectivity}. We recall the following well-known upper bound for the algebraic connectivity of a (possibly weighted) graph. 

\begin{lemma}\label{lemma:bound_algebraic_connectivity}
Let $G$ be a connected graph with $n> 2$ nodes, and let $\lambda_{n-1}(\Delta)$ be the smallest positive eigenvalue of $\Delta$. Then if $\min_{i} (D)_{ii} \leq d$, where
\begin{equation*}
d =
    \begin{cases}
        1 & \text{ if $G$ is unweighted}, \\
        (n-1)/n &\text{ otherwise},
    \end{cases}
\end{equation*}
we have that 
\begin{equation*}
\lambda_{n-1}(\Delta) \leq 1.
\end{equation*}

\begin{proof}
If $G$ is unweighted, this result is \cite[Theorem 5.9]{Das_algebraic_connectivity}. If $G$ is weighted, this result is a consequence of \cite[Proposition 3.5]{Fiedler_algebraic_connectivity}, which shows that $\lambda_{n-1}(\Delta)\leq [n/(n-1)]\min_{i}(D)_{ii}$.  
\end{proof}
\end{lemma}
 Notice that if $G$ is unweighted, Lemma \ref{lemma:bound_algebraic_connectivity} ensures that it is sufficient to check whether $G$ has a \emph{leaf node}, which are nodes that have just one connection, to conclude that $\lambda_{n-1}(\Delta) \leq 1$. Moreover, if $G$ is weighted,  Lemma \ref{lemma:bound_algebraic_connectivity} indicates that it is always possible to rescale the weights to guarantee that $\lambda_{n-1}(\Delta) \leq 1$. \newline

 The graph Laplacian appears in numerous applications and has several useful properties. Notably, it describes how each node can diffuse mass or information to its neighbors, modeling flow or spread. For this reason, the graph Laplacian is widely used to model diffusion processes on networks \cite{Bullo_lectures}.
 For example, let us consider the heat equation with initial distribution $\mathbf{x}_0\in\R^{n}$ on a connected graph $G$:
\begin{equation} \label{eq:simple_heat_equation}
\begin{cases}
    \diff{\mathbf{x}(t)}{t} +\Delta \mathbf{x}(t) =0, & t>0, \\
    \mathbf{x}(0) = \mathbf{x}_0.
\end{cases}
\end{equation}
Then, the solution of \eqref{eq:simple_heat_equation} for all $t\geq0$ can be written as
\begin{equation*}
    \mathbf{x}(t) = e^{-t\Delta} \mathbf{x}_0 = \sum_{k=1}^n e^{-\lambda_k(\Delta)t}(\mathbf{q}_k^T\mathbf{x}_0)\mathbf{q}_k,
\end{equation*} 
where $\lambda_1(\Delta)\geq \dots \geq \lambda_{n-1}(\Delta)> \lambda_n(\Delta)=0$ are the real eigenvalues of $\Delta$ and $\mathbf{q}_1,\ldots,\mathbf{q}_n$ are the corresponding orthogonal and normalized 
eigenvectors. Recalling that $\mathbf{q}_n=(1/\sqrt{n})\mathbf{1}$ and letting $t\rightarrow+\infty$, we have that \begin{equation} \label{eq:convergence_heat_equation}
    \mathbf{x}(t) = \frac{1}{n}(\mathbf{1}^T\mathbf{x}_0)\mathbf{1} + \sum_{k=1}^{n-1} e^{-\lambda_k(\Delta)t}(\mathbf{q}_k^T\mathbf{x}_0)\mathbf{q}_k \longrightarrow \frac{1}{n}(\mathbf{1}^T\mathbf{x}_0)\mathbf{1}=:\overline{\xvec}_0,
\end{equation}
which is exactly the \emph{consensus equilibrium}, that is, the average of the initial distribution $\mathbf{x}_0$. Moreover, it is easy to see that for any connected graph $G$ and any $\mathbf{x}_0\in\R^{n}$, the speed of convergence in \eqref{eq:convergence_heat_equation} is governed by the term $\exp(-\lambda_{n-1}(\Delta)t)$. This means that the convergence of $\xvec(t)$ to $\overline{\xvec}_0$ gets faster as the smallest positive eigenvalue of $\Delta$, i.e., the algebraic connectivity of $G$, increases~\cite{Bullo_lectures}.

\section{The fractional graph Laplacian}\label{sec:fractional_laplacian}

In this section, we recall the definition of the fractional graph Laplacian (Section~\ref{subsec:fractional_graph}) and some of its properties (Section \ref{subsec:properties_fL}). Moreover, we discuss its application in modeling non-local diffusion on graphs (Section \ref{subsec:model_diffusion}).

\subsection{The fractional graph}\label{subsec:fractional_graph}
Let $G$ be a graph ad let $\Delta$ be its Laplacian matrix. Since $\Delta$ is symmetric and semipositive definite, we can write its spectral decomposition $\Delta=Q \Lambda Q^T$, where $Q=[\mathbf{q}_1,\dots,\mathbf{q}_n]\in\R^{n\times n}$ is the unitary matrix of eigenvectors and $\Lambda$ is a diagonal matrix such that $(\Lambda)_{ii}=\lambda_i(\Delta)$ are the non-negative eigenvalues of $\Delta$.
Now, let $\alpha$ be a parameter in $(0,1)$, we define the \emph{fractional graph Laplacian} of $G$ as the matrix
\begin{equation*} 
    \fL := Q \Lambda^{\alpha} Q^T, 
\end{equation*}
where $(\Lambda^{\alpha})_{ii}=(\lambda_i(\fL))=(\lambda_i(\Delta))^\alpha$ for each $i=1,\dots,n$. Observe that, for any given parameter $\alpha\in(0,1)$, $\fL\onevec = \zerovec$ and $(\fL)_{ij} <0$ for each $i\neq j$~\cite{Riascos_introduction_of_fractional_laplacian}. As a consequence, there exists a unique loopless graph such that its associated Laplacian matrix is $\fL$. Such graph, hereafter denoted as \emph{fractional graph}, is defined as $G^\alpha=(\cV,w^{\alpha})$, where
\begin{equation*}
 w^{\alpha}_{ij}=
 \begin{cases}
     -(\fL)_{ij} & \text{ if } i\neq j,\\
     0 & \text{ otherwise,}
 \end{cases}
\end{equation*}
is the edge-weight function.
 
\subsection{Some properties of the fractional graph Laplacian}\label{subsec:properties_fL}

Let us denote by $\fW$ the weight matrix of the fractional graph $\fG$. Since $(\fW)_{ij}=w^{\alpha}_{ij}>0$ for each $i\neq j$, we have that $\fG$ is always a complete weighted graph. Furthermore, it was observed in \cite{Benzi_2020_fractional_laplacian,Bianchi2022} and, most recently, in \cite{Schweitzer_2022_decay_bounds_fL} that the weight of the edge connecting two nodes $i\neq j$ exhibits a power-law decay with respect to the inverse of their combinatorial distance $\dist(i,j)$. In particular, we recall the following asymptotic estimated bound.

\begin{proposition}[\cite{Schweitzer_2022_decay_bounds_fL}] \label{prop:bound_decay_shweitzer}
Let $G$ be a graph and let $\alpha\in(0,1)$. Then, for each $i,j$ in $G$ such that $\dist(i,j)>1$, 
\begin{equation*}
       0< w_{ij}^\alpha\lesssim C \cdot \dist(i,j)^{-2\alpha},
\end{equation*}
where $C$ is a constant independent of $i$,$j$ and $\alpha$.
\end{proposition}

Therefore, according to the above bound, we have that the strength of the non-local interactions in $\fG$ decreases as $\alpha\rightarrow1^{-}$. 
In other words, the larger the parameter $\alpha$, the smaller the difference between the non-local dynamics induced by $\fL$ and the standard (local) dynamics induced by $\Delta$ tends to be. Indeed, the following basic lemma holds.  

\begin{lemma} \label{lemma:limit_fL_alpha_to_1}
For $\alpha\rightarrow 1^-$ we have that, entrywise, $\fL\rightarrow \Delta$.
\begin{proof}
The proof of this lemma comes directly from the fact that, as $\alpha\rightarrow 1^-$, $\lambda_i(\fL)\rightarrow\lambda_i(\Delta)$  for each $i=1,\dots,n-1$.   
\end{proof}
\end{lemma}

On the other hand, when $\alpha$ approaches the left side of the interval $(0,1)$, we have the following lemma.

\begin{lemma} \label{lemma:limit_fL_alpha_to_0}
\begin{enumerate}[label=\normalfont{(\alph*)}]
    \item For $\alpha \rightarrow 0^+$ we have that      
        \begin{equation*}
    \fL  \longrightarrow X:=\frac{1}{n}\begin{bmatrix}
             n-1 & -1 
             & \cdots &-1 \\
             -1 & \ddots & \ddots & \vdots \\
             \vdots & \ddots & \ddots  & -1 \\
             -1& \cdots& -1 & n-1
            \end{bmatrix} =  I - \frac{1}{n}\mathbf{1}\mathbf{1}^T,
               \end{equation*}
               with respect to the spectral norm $\|\cdot\|_2$; and
   \item the convergence rate is governed by the speed at which
\begin{equation*}
     \lambda_{i^*}(\Delta)^\alpha\longrightarrow 1, \text{ as } {\alpha \rightarrow 0^+},
\end{equation*}
where 
\begin{equation*}
i^* =
  \begin{cases}
    1 & \text{ if } \lambda_1(\Delta)> 1/\lambda_{n-1}(\Delta),\\
    n-1 & \text{ otherwise. }
\end{cases}  
\end{equation*}
\end{enumerate}
\end{lemma}
\begin{proof}
Let $(\Lambda,Q)$ the eigenpair of $\Delta$ and let $\lambda_1(\Delta)\geq \dots \geq \lambda_{n-1}(\Delta)>\lambda_{n}(\Delta)=0$ the diagonal entries of $\Lambda$. Note that $X = Q \diag(1, \dots, 1, 0)Q^T$, as the normalized eigenvector of $\Delta$ corresponding to 0 is $(1/\sqrt{n})\mathbf{1}$. Hence
\begin{align*} \|\fL - X\|_2 &= \|\Lambda^\alpha - \diag(1, \dots, 1, 0)\|_2 \\
&=
\max_{i=1,...,n-1}{|\lambda_i(\Delta)^\alpha - 1|} = \max\{|\lambda_1(\Delta)^\alpha - 1|, |\lambda_{n-1}(\Delta)^\alpha
 - 1|\}.
\end{align*}
 The proof of~(a) follows from the fact that, as $\alpha\rightarrow0^+$, $\lambda_i(\Delta)^\alpha\rightarrow 1$ for each $i=1,\dots,n-1$. To prove (b), define $s=\argmin_{i=1,\dots,n}\{\lambda_{i}(\Delta)\leq1\}$ and consider three cases:
\begin{enumerate}
    \item if $s=n$, then $\lambda_{1}(\Delta)\geq\lambda_{n-1}(\Delta)>1$. Hence, the convergence rate is governed by the pace at which $\lambda_{1}(\Delta)^\alpha\rightarrow1$;
    \item if $s=1$, then $1\geq \lambda_{1}(\Delta)\geq\lambda_{n-1}(\Delta)$. Therefore, the convergence rate is determined by the speed at which
    $\lambda_{{n-1}}(\Delta)^\alpha\rightarrow1$;
    \item if $1< s < n$, then $\lambda_1(\Delta)>1>\lambda_{n-1}(\Delta)$.
    For $\alpha\rightarrow0^+$ we have that 
    \begin{align*}
        \lambda_{1}(\Delta)^{\alpha} -1 &\sim \alpha \log \lambda_1(\Delta)\\
        1- \lambda_{n-1}(\Delta)^{\alpha} &\sim -\alpha \log \lambda_{n-1}(\Delta).
    \end{align*}
    Therefore,
    \begin{align*}
        \|\Delta^\alpha-X\|_2\sim \alpha \max\left\{\log \lambda_1(\Delta),\log\frac{1}{\lambda_{n-1}(\Delta)}\right\}, \text{ as } \alpha\rightarrow0^+,
    \end{align*}
    which means that the convergence rate is given by the rate at which $\lambda_{{1}}(\Delta)^\alpha\rightarrow1$ if $\lambda_1(\Delta)>1/\lambda_{{n-1}}(\Delta)$, or $\lambda_{{n-1}}(\Delta)^\alpha\rightarrow1$ otherwise. 
\end{enumerate}
\end{proof}

\begin{remark}
 Note that $\|\Delta^\alpha-X\|_2\rightarrow0,$ for $\alpha\rightarrow0^+$ implies $\fL \rightarrow X$, for $\alpha\rightarrow0^+$ entrywise. Moreover, $X$ is independent of the topology of the graph and it does not depend on the edge-weight function $w$.  
\end{remark}

\subsection{Non-local diffusion via the fractional Laplacian}\label{subsec:model_diffusion}

The fractional Laplacian can be used to model \emph{non-local} diffusion processes on a graph in which mass/information is allowed to make ``long-range jumps'', i.e., travel from a node $i$ to a node $j$ in one step even if they are not connected by an edge. More precisely, let us consider the continuous time process on the fractional graph $\fG$ given by:
\begin{equation}\label{eq:non-local_heat_equation}
\begin{cases}
    \diff{\mathbf{y}}{t} +\fL\mathbf{y}=0, & t>0, \\
    \mathbf{y}(0) = \mathbf{y}_0. 
\end{cases}
\end{equation}
This can be considered as an analogue of \eqref{eq:simple_heat_equation} where the local dynamics induced by $\Delta$ is replaced by the non-local ``superdiffusive'' dynamics induced by $\fL$. The solution of \eqref{eq:non-local_heat_equation} for $t\ge0$ is given by 
\begin{equation*} \label{eq:solution_non-local_heat_equation}
    \mathbf{y}(t) = e^{-\fL t} \mathbf{y}_0 = \frac{1}{n}(\mathbf{1}^T\mathbf{y}_0)\mathbf{1} + \sum_{k=1}^{n-1} e^{-\lambda_k(\fL)t}(\mathbf{q}_k^T\mathbf{y}_0)\mathbf{q}_k,
\end{equation*}
where $\lambda_1(\fL)\geq\dots\geq\lambda_{n-1}(\fL)$ are the positive eigenvalues of $\fL$ and $\mathbf{q}_1,\dots,\mathbf{q}_{n-1}$ are the relative orthogonal and normalized eigenvectors. Moreover, it is easy to see that $\yvec(t)$ converges to the consensus equilibrium $\overline{\yvec}_{0}=(1/n)(\onevec^T\mathbf{y}_0)\onevec$ as $\exp(-\lambda_{n-1}(\fL)t)$.

It is interesting to compare the speed at which $\mathbf{y}(t)$ and $\mathbf{x}(t)$, the solution of~\eqref{eq:simple_heat_equation}, converge to the same consensus equilibrium $\consensus$ if both start with the initial distribution $\xvec_{\zerovec}\in\R^{n}$. Intuitively, one would expect the ability to make long-jumps given by the non-local ``superdiffusive'' operator $\fL$ to speed up the convergence of $\yvec(t)$, when compared to $\xvec(t)$. 
However, as already noted by Estrada in \cite{Estrada_2021_fractional_vs_path_laplacians}, this ultimately depends on the graph, as it could happen that $\lambda_{n-1}(\Delta)\geq\lambda_{n-1}(\fL)$, i.e., the fractional graph $\fG$ has a smaller algebraic connectivity than $G$. In fact, the following lemma holds.
\begin{lemma} \label{lemma:eigenvalues_of_L_and_fL}
    Let $\lambda_i(\Delta)$ and $\lambda_i(\fL)$, the $i$-th eigenvalues of $\Delta$ and $\fL$, and let $s=\argmin_{i=1,\dots,n}\{\lambda_{i}(\Delta)\leq1\}$. Then,
\begin{enumerate}[label=\normalfont{(\roman*)}]
    \item $\lambda_i(\fL) \leq \lambda_i(\Delta)$, for each  $i=1,\dots,s-1$; 
    \item  $\lambda_i(\fL) \geq \lambda_i(\Delta)$,  for each  $i=s,\dots,n-1$;
    \item $\lambda_n(\fL)=\lambda_n(\Delta)=0$.
\end{enumerate}
\end{lemma}
\begin{proof}
The proofs of (i)-(iii) follow from the fact that $\alpha\in(0,1)$, and the fact that $\lambda_i(\fL)=~\lambda_i(\Delta)^{\alpha}$ for each $i=1,\ldots,n$.
\end{proof}

In other words, if $G$ is such that 
$s=n$, by Lemma~\ref{lemma:eigenvalues_of_L_and_fL} the dynamics induced by $\fL$ causes a slower convergence rate of the solution to the equilibrium, when compared to the one induced by $\Delta$. An example of such a graph is presented below.

\begin{example}\label{example:toy_graph}
Consider the graph depicted in Figure \ref{fig:toy_graph}. A simple calculation yields 
\begin{equation*}
\Delta = 
\begin{bmatrix}
3  & -1 & -1 &  0  & -1 \\
-1 &  2 & -1 &  0  &  0 \\
-1 & -1 &  3 & -1  &  0 \\
0  &  0 & -1 &  2  & -1 \\
-1 &  0 &  0 & -1  &  2 \\
\end{bmatrix},\text{ and }
\lambda_{n-1}(\Delta)=1.382.    
\end{equation*}
Table \ref{tab:toy_graph_algebraic_connectivity} shows some values of $\lambda_{n-1}(\fL)$ relative to some choices of $\alpha\in(0,1]$. Since $\lambda_{n-1}(\Delta)\geq1$, we see that $\lambda_{n-1}(\fL)\leq\lambda_{n-1}(\Delta)$ for every selected value of $\alpha\in(0,1]$. 
\end{example}

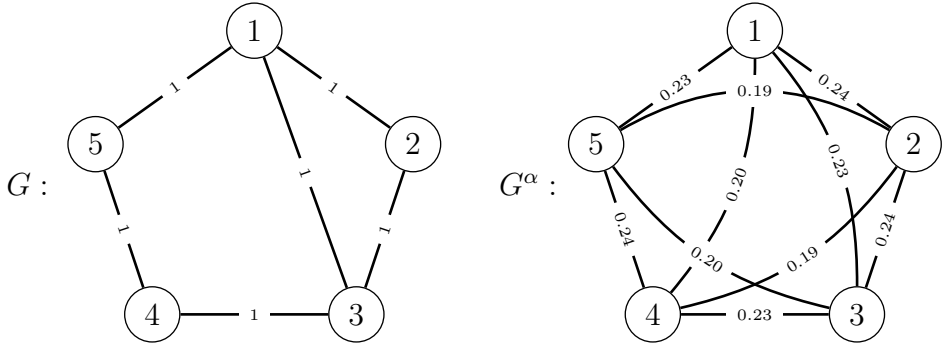
\begin{figure}[thb]
    \centering
\begin{tikzpicture}[scale=1.5]
  \SetUpEdge[lw=1pt, color=black, labelcolor=white]
  \GraphInit[vstyle=Normal] 
  \SetGraphUnit{3}
  \tikzset{VertexStyle/.append style={fill}}
  \tikzset{EdgeStyle/.style={-, sloped, anchor=center}}

\node (G) at (-2,0.65) {$G:$};
   \Vertex[x=0 ,y=2]{1}
   \Vertex[x=1.4 ,y=1]{2}
   \Vertex[x=0.9,y=-0.5]{3}
   \Vertex[x=-0.9 ,y=-0.5]{4}
   \Vertex[x=-1.4 ,y=1]{5}

   \Edge[label=\tiny$1$](1)(2)
   \Edge[label=\tiny$1$](1)(3)
   \Edge[label=\tiny$1$](2)(3)
   \Edge[label=\tiny$1$](3)(4)
   \Edge[label=\tiny$1$](4)(5)
   \Edge[label=\tiny$1$](5)(1)
\end{tikzpicture}
    \hspace{10pt}
\begin{tikzpicture}[scale=1.5]
  \SetUpEdge[lw=1pt, color=black, labelcolor=white]
  \GraphInit[vstyle=Normal] 
  \SetGraphUnit{3}
  \tikzset{VertexStyle/.append style={fill}}
  \tikzset{EdgeStyle/.style={-, sloped, anchor=center}}

\node (G) at (-2,0.65) {$\fG:$};
   \Vertex[x=0 ,y=2]{1}
   \Vertex[x=1.4 ,y=1]{2}
   \Vertex[x=0.9,y=-0.5]{3}
   \Vertex[x=-0.9 ,y=-0.5]{4}
   \Vertex[x=-1.4 ,y=1]{5}

  \Edge[label=\tiny$0.24$](1)(2)

  \Edge[label=\tiny$0.24$](4)(5)
  \Edge[label=\tiny$0.23$](1)(5)
    \Edge[label=\tiny$0.24$](2)(3)

  \Edge[label=\tiny$0.23$](3)(4)
     \tikzset{EdgeStyle/.append style = {bend left=20}}
      \Edge[label=\tiny$0.19$](2)(4)
        \Edge[label=\tiny$0.20$](1)(4)
            \Edge[label=\tiny$0.20$](3)(5)
              \Edge[label=\tiny$0.23$](1)(3)
       \tikzset{EdgeStyle/.style={-, sloped, anchor=center,bend right}}
      \Edge[label=\tiny$0.19$](2)(5)
\end{tikzpicture}

\caption{Graph from Example \ref{example:toy_graph} and the respective fractional graph with $\alpha=0.1$ }
\label{fig:toy_graph}

\end{figure}

\begin{table}[t]
    \centering
        \caption{Algebraic connectivity of $\fG$ from Example \ref{example:toy_graph} when $\alpha$ varies in $(0,1]$}.
    \begin{tabular}{lcccccc}
    \toprule
 $\alpha$ & 0.1 & 0.3 & 0.5 & 0.7 & 0.9 & 1 \\
\midrule
$\lambda_{n-1}(\fL)$ & 1.033 & 1.102 & 1.176 & 1.254 & 1.338 & 1.382 \\
\bottomrule
    \end{tabular}
    \label{tab:toy_graph_algebraic_connectivity}
\end{table}

\section{Regularizing the fractional graph Laplacian} \label{sec:reg_fractional_laplacian}

In this section, after recalling the definition of the regularized fractional graph Laplacian (Section~\ref{subsec:regularized}), we present a method to compute it based on Boolean--Hadamard algebra (Section~\ref{subsec:building_regularized}). We then analyze its spectral properties in detail, with particular emphasis on its relation to the standard and fractional Laplacians (Section~\ref{subsec:spectral_properties_regularized}), and conclude with a discussion on the related computational aspects (Section~\ref{subsec:cost_regularized}).

\subsection{The (scaled) regularized fractional graph}
\label{subsec:regularized}

Given a graph $G=(\cV,w)$ and a parameter $\alpha\in(0,1)$, the \emph{regularized fractional graph} \cite{Bianchi2022} is defined as $\rfG = (\cV,\leftidx{_r}w^\alpha)$, where $\cV$ is the same node set of $G$, and $\leftidx{_r}w^\alpha$ is the nonnegative function that assigns to each pair of nodes $i\neq j$ the following weight:
\begin{equation*}
\leftidx{_r}{w}^{\alpha}_{ij}:=
\begin{cases}
    w_{ij}  & \text{ if } (i,j)\in\cE, \\
    w_{ij}^{\alpha} &  \text{ otherwise, } \\
\end{cases}
\end{equation*}
 where $w_{ij}^{\alpha}$ is defined as in Section~\ref{subsec:fractional_graph}. In other words, the function $\leftidx{_r}{w}^{\alpha}$ gives to the edge $(i,j)$ the same weight as that of $G$ if $i$ and $j$ are already connected in the original graph, while it gives to $(i,j)$ the same weight as the one in $G^\alpha$ otherwise.

Contrary to the fractional graph, the regularized fractional graph preserves by construction the original diffusion dynamics when restricted to the edges of the input graph. However, enforcing structural compatibility could result in a significant disparity between the weights assigned to existing edges and those associated with newly introduced long-range interactions. This imbalance may weaken the non-local exploratory behavior typical of the original (non-compatible) fractional diffusion. To compensate for this attenuation and recover the intended level of non-locality, a scaling parameter 
$\beta\ge1$ could be introduced that amplifies the contribution of the fractional component, counterbalancing the effect of regularization. We refer to \cite{Bianchi2022} for heuristics on how to choose the parameter $\beta$ in such a way that the decay depicted in Proposition \ref{prop:bound_decay_shweitzer} applies as well. 

In other words, one could consider
the \emph{scaled regularized fractional graph} defined as the graph $\srfG = (\cV,\leftidx{_r^\beta}{w}^{\alpha})$, where 
the edge-weight function is given by
\begin{equation*}
\leftidx{_r^\beta}{w}^{\alpha}_{ij}:=
\begin{cases}
    w_{ij}  & \text{ if } (i,j)\in\cE, \\
    \beta w_{ij}^{\alpha} & \text{ otherwise. }\\
\end{cases}
\end{equation*}
Similarly, the scaled regularized fractional Laplacian is defined as follows.
\begin{definition} 
Let $\alpha\in(0,1)$ and let $\beta\geq1$.  We define the \emph{scaled regularized fractional Laplacian} as
     \begin{equation} \label{eq:reg_laplacian_defin}
         \srfL = \diag(\srfW\mathbf{1})-\srfW,
         \end{equation}
         where $\srfW$ is the weight matrix of $\srfG$.
\end{definition}

For simplicity of presentation, we focus on the case where $\beta=1$, and denote by $\rfW$ and $\rfL$ the weight matrix and graph Laplacian of ${}_r^{1}\fG=:\rfG$. We stress that all the results for $\rfG$ that are given in this section can be extended to $\srfG$ for each real parameter $\beta>1$ as well.

\subsection{Building the regularized fractional graph}\label{subsec:building_regularized}
 Let $G$ be a graph with weight matrix $W$ and adjacency matrix $A$, let $\alpha\in(0,1)$, and let $\rfW=(\leftidx{_r}{w}^{\alpha}_{ij})$ be the weight matrix of $\rfG$. The matrix $\rfW$ can be obtained by manually substituting some entries of $W$ with some elements taken from $\fW$ according to the pattern given by $\cE$. However, this approach can be computationally inconvenient if $n$ is large and $G$ is sparse, as it can require $O(n^2)$ substitutions. In the following, we present an explicit formula for $\rfW$, and consequently for $\rfL$, which is useful for further investigations on its spectral properties and offers several computational advantages; see Sections~\ref{subsec:spectral_properties_regularized} and~\ref{subsec:cost_regularized} for more details.

Our formulas exploit the Boolean matrices and their properties. With the purpose of simplifying the presentation, we introduce the following definition. 

\begin{definition}\label{def:Asum}
    Let $A$ be a $n \times n$ Boolean matrix and let $U,V\in \R ^{n \times n}$. The \emph{$A$-sum} of $U$ and $V$ is the $n \times n$ matrix given by
    \begin{equation*} \label{eq:A-sum}
     U +_A V := U \circ A + V \circ A^{C},   
    \end{equation*}
    where $A^{C}$ is the complementary matrix of $A$.
\end{definition}

Some properties of this matrix operation are given in the Propositions \ref{prop:build_Z_with_A_sum}-\ref{prop:A-sum formula}. 

\begin{proposition} \label{prop:build_Z_with_A_sum}
Let $U=(u_{ij}),V=(v_{ij})$ in $\R^{n \times n}$ and let $\cS$ be a binary relation over the set $\{1,\dots,n\}$. Moreover, let $Z$ be such that 
\begin{equation*}
    (Z)_{ij} =z_{ij} :=
    \begin{cases}
        u_{ij} & \text{if } (i,j)\in\cS,  \\
        v_{ij} & \text{if } (i,j)\notin\cS.
    \end{cases}
\end{equation*}
Then, 
\begin{equation*}
    Z  = U +_{A^{}}V,
\end{equation*}    
where $A$ is the boolean matrix representing ${\cS}$ given by
\begin{equation*}
(A^{})_{ij}= a_{ij}:=
    \begin{cases}
        1 & \text{ if } (i,j)\in \cS, \\
        0 & \text{ if } (i,j)\notin \cS.
    \end{cases}
\end{equation*}
\end{proposition}
\begin{proof}
From the definition of $Z$ and $A$ we have that, entrywise,
\begin{equation*}
    (Z)_{ij}=z_{ij} = u_{ij}a^{}_{ij} + v_{ij}(1- a^{}_{ij}).
\end{equation*}
Therefore, $Z= U\circ A^{} + V \circ (\onevec\onevec^T - A^{})$. The fact that $A^C=\onevec\onevec^T - A^{}$ yields the conclusion.
\end{proof}

\begin{proposition}\label{prop:A-sum formula} In the notation of Definition \ref{def:Asum},
\begin{equation} \label{eq:A-sum formula}
    U +_A V = V + (U-V) \circ A,
\end{equation}
and, equivalently,
\begin{equation} \label{eq:A-sum formula complement}
    U +_A V = U + (V-U) \circ A^C.
\end{equation}
\end{proposition}
\begin{proof} The proof of \eqref{eq:A-sum formula} immediately follows from
    \begin{align*}
        U +_A V  &= U \circ A + V \circ \mathbf{1}\mathbf{1}^T -  V \circ A\\
        & = V + (U-V) \circ A.
    \end{align*}
The proof of (\ref{eq:A-sum formula complement}) is identical and is therefore omitted. 
\end{proof}

Thanks to the previous propositions, we can provide an
explicit formula for the weight matrix of the regularized fractional graph in terms of the weight and the adjacency matrices of the original graph and of the fractional graph Laplacian.

\begin{proposition}\label{prop:build_reg_adj_matrix}
For each $\alpha\in(0,1)$ we have that
    \begin{equation}\label{eq:reg_adj_matrix}
        \rfW = \fL \circ (I+A) - \fL +W.
    \end{equation}
    \begin{proof}
    Since $A$ can be seen as the Boolean matrix that represents $\cE$ as a relation over the set $\{1,\dots,n\}$, we can apply Proposition \ref{prop:build_Z_with_A_sum} to get 
    \begin{equation*}
    \rfW = W +_A \fW. 
    \end{equation*}
Now, by using \eqref{eq:A-sum formula} and remembering that $\fW = \diag(\fL) - \fL =  \fL \circ I - \fL $, 
we have 
        \begin{align*}
        \rfW &= \fW +(W - \fW) \circ A   \\
            &= \fW + W\circ A  - \fW \circ A   \\
            & = \fL \circ I - \fL  +W + \fL \circ A - \diag(\fL) \circ A \\
            & =  \fL \circ (I+A) - \fL +W  - \diag(\fL) \circ A .
        \end{align*}
The fact that $\diag(\fL) \circ A=0$ since $A$ is hollow yields the conclusion.
    \end{proof}
\end{proposition} 

Before stating our formula for the regularized fractional graph Laplacian $\rfL$, we introduce the following definition.

\begin{definition} \label{def:restr_f_laplacian}
    Let $A$ be the adjacency matrix of a graph $G$ and let $\alpha\in(0,1)$. Then, the \emph{restricted fractional graph Laplacian} is defined as 
    \begin{equation}\label{eq:restr__fractional_laplacian}
     \restr(\fL):= \fL \circ A - \diag(\left[\fL \circ A\right] \mathbf{1}),  
    \end{equation}
    that is, the graph Laplacian of the graph obtained from $\fG$ by eliminating all edges that are not in $G$. 
\end{definition} 

\begin{theorem}\label{thm:reg_frac_lap}
For each $\alpha\in(0,1)$ we have that 
\begin{equation} \label{eq:reg_laplacian_formula}
        \rfL = \Delta + \Delta^{\alpha} - \restr(\fL).
    \end{equation}
    \begin{proof}
 By definition, we know that
     \begin{equation*} 
         \rfL = \diag(\rfW\mathbf{1})-\rfW.
     \end{equation*}
     Then, by applying Proposition \ref{prop:build_reg_adj_matrix} we get
    \begin{align*}
     \rfL & = \diag((\fL \circ (I+A) - \fL +W)\onevec)-(\fL \circ (I+A) - \fL +W)\\
     &=  \diag(\fL \circ I) + \diag(\left[\fL \circ A\right] \mathbf{1})   - \diag(\fL\mathbf{1}) + \diag(W\onevec) \\
     & \quad - \fL\circ I - \fL \circ A + \fL -W  \\
     &= \diag(\left[\fL \circ A\right] \mathbf{1}) - \diag(\fL\mathbf{1}) - \fL \circ A  + \fL + \Delta \\
     &= \Delta  + \fL - \restr(\fL)   - \diag(\fL\mathbf{1}).  
    \end{align*}
    The fact that $\diag(\fL\mathbf{1})=0$ for each $\alpha\in(0,1)$ yields the conclusion.
    \end{proof}
\end{theorem} 

\subsection{Spectral properties of the regularized fractional Laplacian}\label{subsec:spectral_properties_regularized}

In this section, we study the relation between the spectrum of the regularized fractional Laplacian $\rfL$ and the spectra of $\Delta$ and $\fL$. 

\begin{proposition} \label{prop:shift_eigenvalue_L}
Let $\alpha\in(0,1)$ and let $\lambda_1(\rfL)\geq \dots \geq\lambda_{n-1}(\rfL)$ be the positive eigenvalues of $\rfL$. Then, for all $i=1,\dots,n-1$,
\begin{equation*} \label{eq:eigs_reg_frac_lap}
  \lambda_i(\Delta) + \lambda_1(E_{\alpha}) \geq \lambda_i(\rfL) \geq \lambda_i(\Delta),   
\end{equation*}
where $\lambda_1(\Delta)\geq\dots\geq \lambda_{n-1}(\Delta)$ are the positive eigenvalues of $\Delta$, and $\lambda_1(E_{\alpha})$ is the largest eigenvalue of $E_{\alpha}:=\fL - \restr(\fL)$. Moreover,
\begin{equation*} \label{eq:eigs_frac_lap_limits}
    \lambda_i(\rfL) \rightarrow \lambda_i(\Delta), \, \text{ for each }\, i=1,\dots,n-1, \, \text{ when } \,\alpha\rightarrow1^-.
\end{equation*}
\begin{proof}
It is easy to see that $E_{\alpha}$ is the Laplacian matrix of the graph obtained from $G^{\alpha}$ by removing all the edges in $G$, hence it is symmetric and semipositive definite. From Theorem~\ref{thm:reg_frac_lap} we know that $\rfL=\Delta + E_{\alpha}$, and therefore, a simple consequence of Weyl's inequality (see, e.g., \cite[Example 7.5.2]{meyer2000matrix}) gives us
\begin{equation*}
  \lambda_i(\Delta) + \lambda_1(E_{\alpha}) \geq \lambda_i(\rfL) \geq \lambda_i(\Delta) + \lambda_n(E_{\alpha}),    
\end{equation*}
where $\lambda_1(E_{\alpha}),\, \lambda_n(E_{\alpha})$ are the largest and the smallest eigenvalue of $E_{\alpha}$, respectively. Observing that $\lambda_n(E_{\alpha})=0$ and that, entrywise, $E_{\alpha}\rightarrow 0$ when $\alpha\rightarrow1^-$, yields the conclusion. 
\end{proof}
\end{proposition}

\begin{remark}
The spectrum of $\rfL$ exhibits the same behavior as the spectrum of $\fL$ when $\alpha\rightarrow1^-$. In fact, as guaranteed by Lemma \ref{lemma:limit_fL_alpha_to_1} and Proposition \ref{prop:shift_eigenvalue_L}, both spectra converge (elementwise) to the spectrum of $\Delta$ as $\alpha\rightarrow1^-$.     
\end{remark}

\begin{proposition}\label{prop:shift_eigenvalue_fL}
Let $G$ be a graph with $n$ nodes and let $\epsilon\geq0$ such that
\begin{equation} \label{eq:hypothesis_on_the_Laplacian}
    (\Delta)_{ij}  + \,\epsilon< - \frac{1}{n}  , \text{ for each $(i,j)\in\cE$.} 
\end{equation}
Then there exists an $\alpha_\epsilon\in(0,1)$ such that, for all $i=1,\dots,n-1$, and for all $\alpha\in(0,\alpha_{\epsilon})$, 
\begin{equation*}
  \lambda_i(\rfL) \geq \lambda_i(\fL),   
\end{equation*}
where $\lambda_i(\rfL)$ and $\lambda_i(\fL)$, are the $i$-th eigenvalues of $\rfL$ and $\fL$, respectively.
\begin{proof}
By Theorem \ref{thm:reg_frac_lap} we know that $\rfL = \fL + \Delta - \restr(\fL)$ for all $\alpha\in(0,1)$. Letting $P_{\alpha}:=\Delta - \restr(\fL)$, we can use the same reasoning of the proof of Proposition \ref{prop:shift_eigenvalue_L} to get
\begin{equation*}
  \lambda_i(\rfL) \geq \lambda_i(\fL) + \lambda_n(P_{\alpha}), \text{ for each } i=1,\dots,n-1,      
\end{equation*}
where $\lambda_n(P_{\alpha})$ is the smallest eigenvalue of $P_\alpha$. Notice that $P_\alpha\mathbf{1}= \Delta\onevec - \restr(\fL)\onevec =\mathbf{0}$. In the remaining, assuming \eqref{eq:hypothesis_on_the_Laplacian}, we will prove that there exists an $\alpha_{\epsilon}\in(0,1)$ such that $(P_\alpha)_{ij}\leq0$ for each $i\neq j$. This will allow us to conclude that $\lambda_n(P_{\alpha})=0$ for any $\alpha \in (0,\alpha_{\epsilon})$. By recalling that $\restr(\fL) = \fL \circ A - \diag(\left[\fL \circ A\right] \mathbf{1})$ we have that, for each $i\neq j$,
\begin{equation*}
(P_{\alpha})_{ij} =
    \begin{cases}
    (\Delta)_{ij} - (\fL \circ A)_{ij} & \text{if $(i,j)\in\cE$},   \\
    0 & \text{ otherwise.}
    \end{cases}
\end{equation*}
Now, since for all $(i,j)\in \cE$,
\begin{equation*}
    \lim_{\alpha \rightarrow 0^+}  (\fL \circ A)_{ij} =  (X\circ A)_{ij} = - \frac{1}{n},
\end{equation*}
where $X$ is the matrix in Lemma \ref{lemma:limit_fL_alpha_to_0}, there exists an $\alpha_{\epsilon}\in (0,1)$ such that, for all $(i,j)\in\cE$ and for all $\alpha\in(0,\alpha_{\epsilon})$,
\begin{equation*}
  (P_{\alpha})_{ij}=(\Delta)_{ij} -(\fL\circ A)_{ij}  < (\Delta)_{ij} + \frac{1}{n} +  \epsilon.
\end{equation*}
From this, we can use \eqref{eq:hypothesis_on_the_Laplacian} to conclude that $(P_\alpha)_{ij}<0$ for each  $(i,j)\in \cE$, thus getting the proof.
\end{proof}
\end{proposition}

\begin{remark}
   If $G$ is unweighted, then \eqref{eq:hypothesis_on_the_Laplacian} in Proposition \ref{prop:shift_eigenvalue_fL} is satisfied for all $0\leq\epsilon<(n-1)/n$ since $(\Delta)_{ij} = -1$ for each $(i,j)\in\cE$.
\end{remark}

We are now ready to prove the following theorem, which can be obtained by combining the previous two propositions.

\begin{theorem}\label{thm:eigenvalues_of_L_fL_and_rfL}
Let $G$ and $\epsilon\geq0$ as in the hypotheses of Proposition \ref{prop:shift_eigenvalue_fL}, and let $\lambda_i(\Delta)$, $\lambda_i(\fL)$, $\lambda_i(\rfL)$, be the $i$-th eigenvalues of $\Delta$, $\fL$, and $\rfL$, respectively. Moreover, let $s=\argmin_{i=1,\dots,n}\{\lambda_{i}(\Delta)\leq1\}$. Then, there exists an $\alpha_{\epsilon}\in(0,1)$ such that, for each $\alpha\in(0,\alpha_\epsilon)$, we have
\begin{enumerate}[label=\normalfont{(\roman*)}]
    \item $\lambda_i(\rfL) \geq \lambda_i(\Delta) \geq \lambda_i(\fL)$, for each $i=1,\dots,s-1$; 
    \item $\lambda_i(\rfL) \geq \lambda_i(\fL) \geq \lambda_i(\Delta)$ ,  for each $i=s,\dots,n-1$;
    \item $\lambda_n(\rfL)=\lambda_n(\fL)=\lambda_n(\Delta)=0$.
\end{enumerate}
\end{theorem}
\begin{proof}
(i) is a direct consequence of Lemma \ref{lemma:eigenvalues_of_L_and_fL} and Proposition \ref{prop:shift_eigenvalue_L}, therefore it holds for all $\alpha\in(0,1)$. (ii) is a consequence of Lemma \ref{lemma:eigenvalues_of_L_and_fL} and Proposition \ref{prop:shift_eigenvalue_fL}. (iii) follows from the definitions of $\fL$ and $\rfL$.
\end{proof}

Moreover, we have the following proposition.
\begin{proposition} \label{prop:alg_conn_L_fL_and_rfL}
Let $G$ be a graph such that $\lambda_{n-1}(\Delta)\geq1$. Then, $\lambda_{n-1}(\rfL) \geq \lambda_{n-1}(\Delta) \geq \lambda_{n-1}(\fL)$ for any choice of $\alpha\in(0,1)$.
\end{proposition}
\begin{proof}
This result follows from the fact that by Proposition~\ref{prop:shift_eigenvalue_L}, we have the inequality
$\lambda_{n-1}({}_r\Delta^{\alpha})\geq \lambda_{n-1}(\Delta), \, \forall \alpha\in(0,1)$. Additionally, since $\lambda_{n-1}(\Delta)\geq1$, Lemma~\ref{lemma:eigenvalues_of_L_and_fL}
with $s = n$ gives $\lambda_{n-1}(\Delta^{\alpha}) \leq \lambda_{n-1}(\Delta)$.
\end{proof}

The take-home message of Theorem \ref{thm:eigenvalues_of_L_fL_and_rfL} and Proposition~\ref{prop:alg_conn_L_fL_and_rfL} is the following. For any graph $G$ the spectrum of $\rfL$ corresponds to a ``rightward shift'' of the spectrum of $\Delta$ for any choice of $\alpha\in(0,1)$. The same relation holds between the spectrum of $\rfL$ and the spectrum of $\fL$, although this is true only under certain hypotheses and under certain choices of the parameter $\alpha$. More precisely, this means that the algebraic connectivity of $\rfG$ is always greater than the algebraic connectivity of $G$ for all $\alpha\in(0,1)$, and, under specific conditions, possibly greater than the algebraic connectivity of $\fG$. Therefore, by replacing $\fL$ with $\rfL$ in \eqref{eq:non-local_heat_equation}, we are always in the presence of a superdiffusive behavior compared to the standard heat equation~\eqref{eq:simple_heat_equation}, in contrast to what happens with \eqref{eq:non-local_heat_equation} where the superdiffusion property of the operator $\fL$ fades when $\lambda_{n-1}(\Delta)\geq1$; see Section \ref{subsec:experiments_non_local_diffusion} for a numerical confirmation. 

\begin{remark}
One might think that rescaling the edge weights to guarantee $\lambda_{n-1}(\Delta)<1$ would recover the superdiffusive behavior of the fractional Laplacian. However, this approach brings some inconsistencies. Indeed, rescaling the weights is equivalent to consider two different heat equations, i.e., \eqref{eq:simple_heat_equation} with $c\Delta$ in place of $\Delta$ and \eqref{eq:non-local_heat_equation} with $c^\alpha\fL$ in place of $\fL$, $c\neq1$, whose solutions are no longer comparable.
\end{remark}

We conclude this section by reconsidering the graph from Example \ref{example:toy_graph}, which is depicted next to its regularized fractional graph with $\alpha=0.1$ in Figure \ref{fig:toy_graph_regularized}. As shown in Table \ref{tab:toy_graph_algebraic_connectivity_regularized}, the algebraic connectivity of $\rfG$ is greater than the one of $G$ and $\fG$, as expected, since $\lambda_{n-1}(\Delta)=1.382$ in this example.

\begin{figure}[h] 
    \centering
\begin{tikzpicture}[scale=1.5]
  \SetUpEdge[lw=1pt, color=black, labelcolor=white]
  \GraphInit[vstyle=Normal] 
  \SetGraphUnit{3}
  \tikzset{VertexStyle/.append style={fill}}
  \tikzset{EdgeStyle/.style={-, sloped, anchor=center}}

\node (G) at (-2,0.65) {$G:$};
   \Vertex[x=0 ,y=2]{1}
   \Vertex[x=1.4 ,y=1]{2}
   \Vertex[x=0.9,y=-0.5]{3}
   \Vertex[x=-0.9 ,y=-0.5]{4}
   \Vertex[x=-1.4 ,y=1]{5}

   \Edge[label=\tiny$1$](1)(2)
   \Edge[label=\tiny$1$](1)(3)
   \Edge[label=\tiny$1$](2)(3)
   \Edge[label=\tiny$1$](3)(4)
   \Edge[label=\tiny$1$](4)(5)
   \Edge[label=\tiny$1$](5)(1)
\end{tikzpicture}
    \hspace{10pt}
\begin{tikzpicture}[scale=1.5]
  \SetUpEdge[lw=1pt, color=black, labelcolor=white]
  \GraphInit[vstyle=Normal] 
  \SetGraphUnit{3}
  \tikzset{VertexStyle/.append style={fill}}
  \tikzset{EdgeStyle/.style={-, sloped, anchor=center}}

\node (G) at (-2,0.65) {$\rfG:$};
   \Vertex[x=0 ,y=2]{1}
   \Vertex[x=1.4 ,y=1]{2}
   \Vertex[x=0.9,y=-0.5]{3}
   \Vertex[x=-0.9 ,y=-0.5]{4}
   \Vertex[x=-1.4 ,y=1]{5}

  \Edge[label=\tiny$1$](1)(2)
  \Edge[label=\tiny$1$](1)(3)

  \Edge[label=\tiny$1$](4)(5)
  \Edge[label=\tiny$1$](1)(5)
    \Edge[label=\tiny$1$](2)(3)

  \Edge[label=\tiny$1$](3)(4)

       \tikzset{EdgeStyle/.style={-, sloped, anchor=center,bend right=25}}
      \Edge[label=\tiny$0.19$](2)(5)
              \Edge[label=\tiny$0.20$](1)(4)
                   \Edge[label=\tiny$0.19$](2)(4)
      \tikzset{EdgeStyle/.style={-, sloped, anchor=center,bend left=20}}
            \Edge[label=\tiny$0.20$](3)(5)
            \tikzset{EdgeStyle/.style={-, sloped, anchor=center}}
              \Edge[label=\tiny$1$](1)(3)
\end{tikzpicture}

\caption{Graph from Example \ref{example:toy_graph} and the corresponding regularized fractional graph for $\alpha=0.1$}
\label{fig:toy_graph_regularized}

\end{figure}
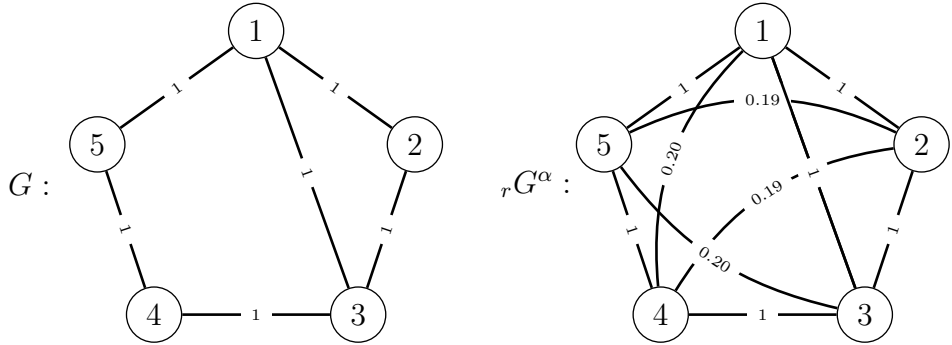

    \begin{table}[t!]
    \centering
        \caption{Algebraic connectivity of $\fG$ and $\rfG$ of the graph from Example \ref{example:toy_graph} when $\alpha$ varies in $(0,1]$}
    \begin{tabular}{lcccccc}
    \toprule
 $\alpha$ & 0.1 & 0.3 & 0.5 & 0.7 & 0.9 &  1 \\
\midrule
$\lambda_{n-1}(\fL)$ & 1.033 & 1.102 & 1.176 & 1.254 & 1.338 &  1.382\\
$\lambda_{n-1}(\rfL)$& 2.089 &2.033 & 1.934 & 1.777 & 1.540  &  1.382 \\
\bottomrule
    \end{tabular}

    \label{tab:toy_graph_algebraic_connectivity_regularized}
\end{table}

\subsection{Computational cost} \label{subsec:cost_regularized}

In this subsection, we estimate the computational cost of computing $\rfW$ and $\rfL$ and show how to exploit the sparsity of $A$ to make our machinery more efficient.

 We recall that, given a sparse $n\times n$ matrix $A$, computing either $M\circ A$ or  $M+A$ with $M\in\R^{n \times n}$, requires a number of operations of order $n$. As a consequence, building $\rfW$ by means of (\ref{eq:reg_adj_matrix}) avoids to explicitly form $A^C$, which is instead very dense (see Remark \ref{rmk:dense_graphs}) and keeps the asymptotic cost of computing $\rfW$ and $\rfL$ as the one of computing $\fL$; see the following propositions.  

\begin{proposition} \label{prop:comp_cost} In the above notation, and assuming the sparsity of $A$, the computational cost of building $\rfW$ with (\ref{eq:reg_adj_matrix}) is asymptotically the same of computing $\fL$ alone.  
\end{proposition}
    \begin{proof} 
    Computing $\fL$ requires computing the spectral decomposition of $\Delta$, which is known to be an operation of order $n^3$ for any symmetric $n\times n$ matrix \cite{Golub_matrix_computation_4_ed}. Applying \eqref{eq:reg_adj_matrix} for building $\rfW$ requires the computation of $\fL$, one Hadamard product with $(I+A)$ and two matrix sums involving two sparse matrices. As the latter have complexity of order $n$, the overall cost for computing $\rfW$ by \eqref{eq:reg_adj_matrix} remains asymptotically the same of computing the fractional Laplacian alone.
    \end{proof}

Notice that, when $G$ has almost $n^2$ edges, $A$ is very dense. However, thanks to Remark~\ref{rmk:dense_graphs}, $A^C$ turns out to be sparse and a similar result given in Proposition~\ref{prop:comp_cost} holds. Indeed, it is sufficient to use (\ref{eq:A-sum formula complement}) instead of \eqref{eq:A-sum formula} in the proof of Proposition~\ref{prop:build_reg_adj_matrix} and exploit the sparsity of $A^C$ in the computation of $\rfW$.

\begin{proposition}\label{prop:comp_cost_laplacian} Under the same hypotheses of Proposition \ref{prop:comp_cost}, the computational cost of computing $\rfL$ using \eqref{eq:reg_laplacian_formula} is asymptotically the same of computing $\fL$ alone.  
\begin{proof} 
If $A$ is sparse, so are $\Delta$ and $\restr(\fL)$. By its definition given in \eqref{eq:restr__fractional_laplacian}, the latter requires the computation of $\fL$ plus few extra operations of order $n$. Therefore, the asymptotic cost of computing $  \Delta+\fL - \restr(\fL)$  is the same of computing $\fL$ alone.  
 \end{proof}
\end{proposition} 

\begin{remark}
Note that the claims of Propositions~\ref{prop:comp_cost}~and~\ref{prop:comp_cost_laplacian} do not change even if, instead of computing the full spectral decomposition, $\Delta^{\alpha}$ is approximated by a few dominant eigenpairs of $\Delta$ using Krylov methods with cost $O(n^2)$.    
\end{remark}

\section{Numerical examples} \label{sec:numerical_examples}

In this section, we test the theoretical results discussed in the previous sections on some synthetic and real-world networks.  

All the experiments are run in Matlab (R2023b) on a laptop equipped with 16 GB of RAM and an Intel i7-1260P CPU operating at 2.10 GHz.

\subsection{Dataset description}

The real-world networks we considered for our experiments can be divided into two subgroups according to their respective sources.

The first subgroup consists in a set of three graphs taken from the SuiteSparse Matrix Collection \cite{SuiteSparse}. In particular, we selected the following networks: \texttt{karate}, an unweighted graph that represents the friendships connection between the members of a university karate club; \texttt{YaleA}, a weighted graph from a machine learning dataset belonging to the ``ML\_Graph'' group; and \texttt{netscience}, a weighted graph that represents the collaborations of scientists working on network theory with edge-weights assigned as described by Newman in \cite{Newman2000ScientificCN}.

The second subgroup consists in two graphs of human connectomes extracted from the ``86 nodes'' folder within the \url{http://braingraph.org} dataset \cite{braingraph}. Graphs in this dataset are generated from MRI (Magnetic Resonance Imaging) scan taken from the Human Connectome Project \cite{Human_connectone} and are built as follows. First the brain is divided into $n$ regions, which serve as graph nodes\footnote{In this case, $n=86$.}. 
Each pair of nodes is then connected by an edge with a weight that is equal to the number of fibers connecting the related regions. The full dataset contains the data of $1064$ subjects and we chose the connectomes of subjects with ID 333330 
and 987074. In the following, we will denote these graphs connectomes by \texttt{ghc1},  
and \texttt{ghc2}, respectively.

Notice that while all real-world networks described in this section are undirected, some of them are not connected. For such instances, we focus solely on the largest connected component. 
Some properties of the networks in the dataset, including the number of nodes $n$, the number of edges $m$, the smallest degree $d_{\min}=\min_{i} (D)_{ii}$, the diameter diam, and the algebraic connectivity $\lambda_{n-1}(\Delta)$, are reported in Table \ref{tab:real_networks_info}. 

\begin{table}[t]
    \centering
    \caption{Some properties of the real-world networks in the dataset: number of nodes $n$, number of edges $m$, network diameter diam, smallest degree $d_{min}$, and  algebraic connectivity $\lambda_{n-1}(\Delta)$}.
    \begin{tabular}{lcccccc}       
		\toprule
		Network & $n$ & $m$ & weighted & $d_{\min}$ & diam & $\lambda_{n-1}(\Delta)$ \\
		\midrule
                  
        \texttt{karate}  & 34 & 74 & no &  1.00  & 5 & 0.47  \\
       \texttt{YaleA} &  165  & 1134 & yes & 0.21  & 5 &  0.03 \\
       \texttt{netscience}  & 379 & 914 & yes & 1.00 & $17$ &    0.01 \\
       \texttt{ghc1} &     75  &  903  &  yes &   7.75 &   4 &      7.81    \\
       \texttt{ghc2} & 72  &  790  & yes &   1.50  & 4  &   1.52  \\

        \bottomrule
    \end{tabular}   
    \label{tab:real_networks_info}
\end{table}

\subsection{Spectra of the (regularized) fractional Laplacian}

In Figure \ref{fig:examples_spectra}, we plot the spectra of $\Delta$, $\fL$, and $\rfL$ for the \texttt{netscience} and \texttt{ghc1} networks, evaluated at $\alpha = 0.9, 0.5, 0.1$.

For the \texttt{netscience} network (Figure \ref{subfig:netscience_fL} and \ref{subfig:netscience_rfL}), we have that all nonzero eigenvalues of $\fL$ tend to one as $\alpha$ approaches zero. In particular, eigenvalues smaller than one increase towards one, while those greater than one exhibit the opposite behavior. Meanwhile, the spectra of $\rfL$ correspond to a ``right-shift'' of the spectrum of $\Delta$ for all chosen values of $\alpha$. These observations are in line with what predicted by Lemma~\ref{lemma:eigenvalues_of_L_and_fL} and Proposition \ref{prop:shift_eigenvalue_L}.  

For the \texttt{ghc1} network (Figure~\ref{subfig:ghc1_fL} and \ref{subfig:ghc1_rfL}), we find that $\lambda_{n-1}(\Delta)>1$. As a consequence, we have that all nonzero eigenvalues of $\fL$ decrease in magnitude, tending toward one, while all eigenvalues of $\rfL$ consistently remain uniformly above the eigenvalues of $\Delta$. These results are consistent with what we expect from Proposition~\ref{prop:alg_conn_L_fL_and_rfL}. 

Notice that, unlike $\lambda_{n-1}(\Delta^{\alpha})$, $\lambda_{n-1}({}_r\Delta^{\alpha})$ is not generally monotonic with respect to $\alpha$. In fact, we can see from Figure \ref{subfig:netscience_rfL} and \ref{subfig:ghc1_rfL} that $\lambda_{n-1}({}_r\Delta^{0.1})>\lambda_{n-1}({}_r\Delta^{0.9})>\lambda_{n-1}(\Delta)$ for the \texttt{netscience} network, while $\lambda_{n-1}({}_r\Delta^{0.9})>\lambda_{n-1}({}_r\Delta^{0.1})>\lambda_{n-1}(\Delta)$ for the \texttt{ghc1} network. 

\begin{figure}[th]
    \centering
    \begin{subfigure}[b]{0.40\linewidth}
        \centering
        \includegraphics[width=\linewidth]{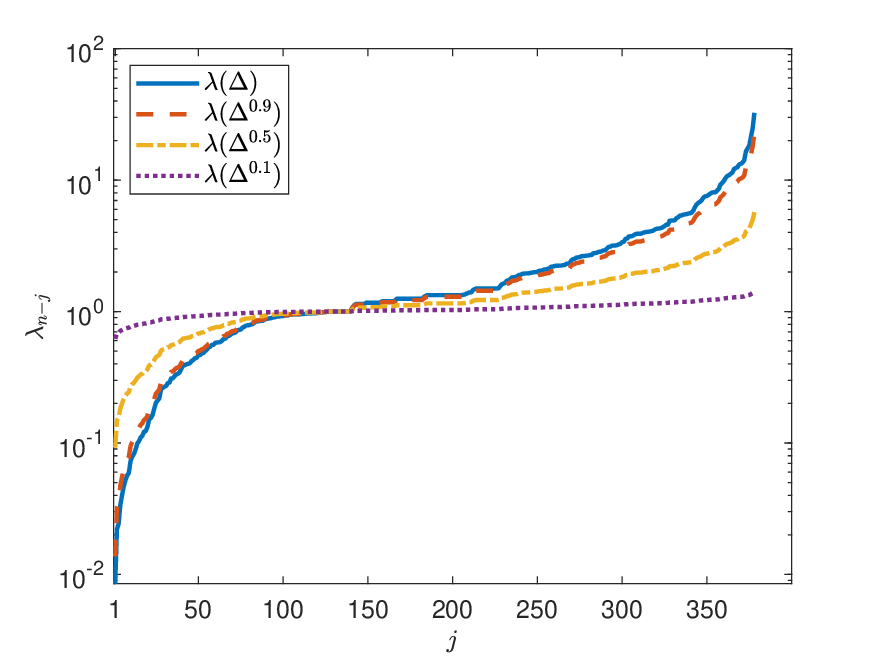}
        \caption{$\lambda(\fL)$, \texttt{netscience} network}
        \label{subfig:netscience_fL}
    \end{subfigure}
    \begin{subfigure}[b]{0.40\linewidth}
        \centering
        \includegraphics[width=\linewidth]{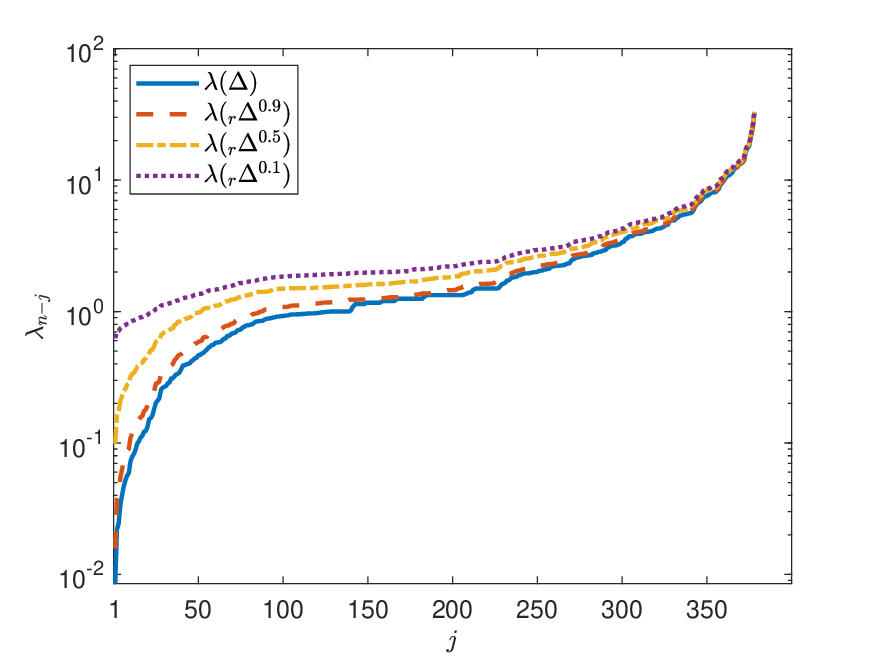}
        \caption{$\lambda(\rfL)$, \texttt{netscience} network}
        \label{subfig:netscience_rfL}
    \end{subfigure}
    \begin{subfigure}[b]{0.40\linewidth}
        \centering
        \includegraphics[width=\linewidth]{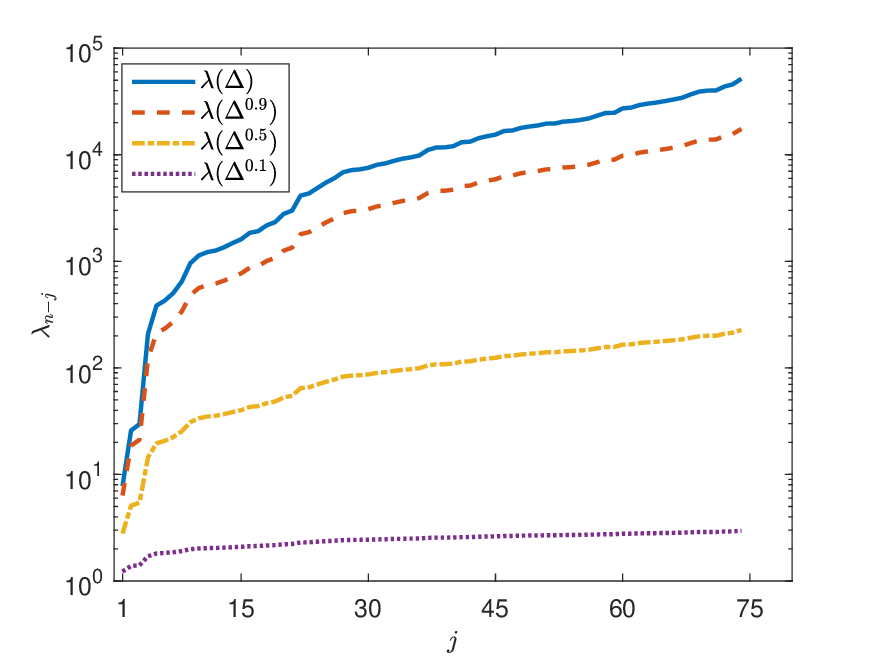}
        \caption{$\lambda(\fL)$, \texttt{ghc1} network}
        \label{subfig:ghc1_fL}
    \end{subfigure}
    \begin{subfigure}[b]{0.40\linewidth}
        \centering
        \includegraphics[width=\linewidth]{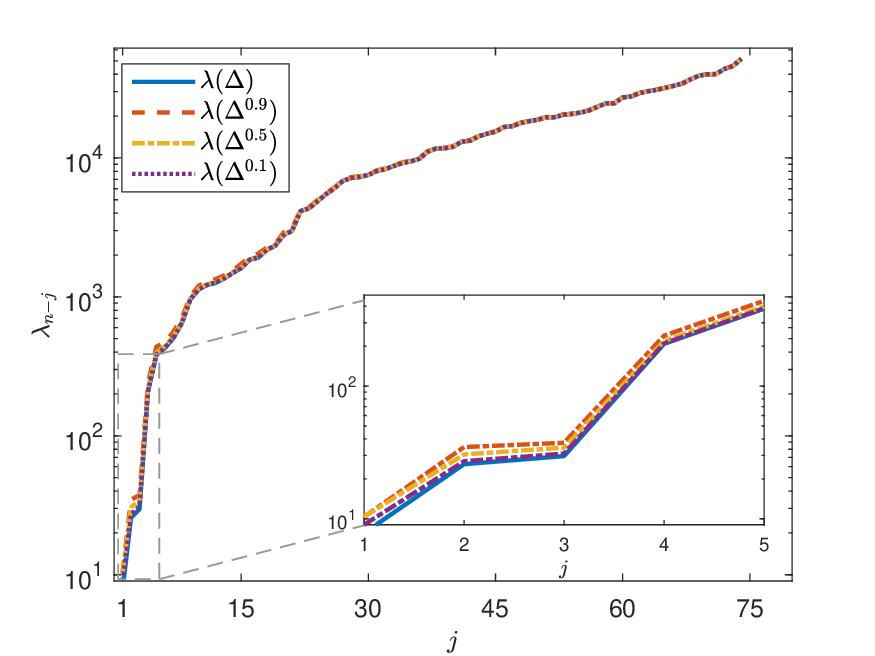}
        \caption{$\lambda(\rfL)$, \texttt{ghc1} network}
        \label{subfig:ghc1_rfL}
    \end{subfigure}
    \caption{ Spectra of $\Delta$, $\fL$ and $\rfL$ related to the \texttt{netscience} and  \texttt{ghc1} networks for $\alpha=0.9,0.5,0.1$ 
    }
    \label{fig:examples_spectra}
\end{figure}

Now, let us focus on Figure \ref{fig:spectra_that_switch_behavior}, which shows the first thirty nonzero eigenvalues of $\fL$ and $\rfL$ for the \texttt{YaleA} network, evaluated at $\alpha=0.60,0.40,0.20$. Note that for $\alpha=0.40,0.20$, all plotted eigenvalues of $\rfL$ are uniformly greater than the eigenvalues of $\fL$, while this does not hold for $\alpha=0.60$. 
These results are expected from Proposition~\ref{prop:shift_eigenvalue_fL}. The \texttt{YaleA} network, in fact, satisfies the hypothesis \eqref{eq:hypothesis_on_the_Laplacian} for all $0\leq \epsilon <  2.1\cdot 10^{-3}$ since we have that $(\Delta)_{ij}~\leq~-8.2\cdot~10^{-3}$ for all $(i,j)\in\cE$.  

\begin{figure}[th]
    \centering
    \includegraphics[width=\linewidth]{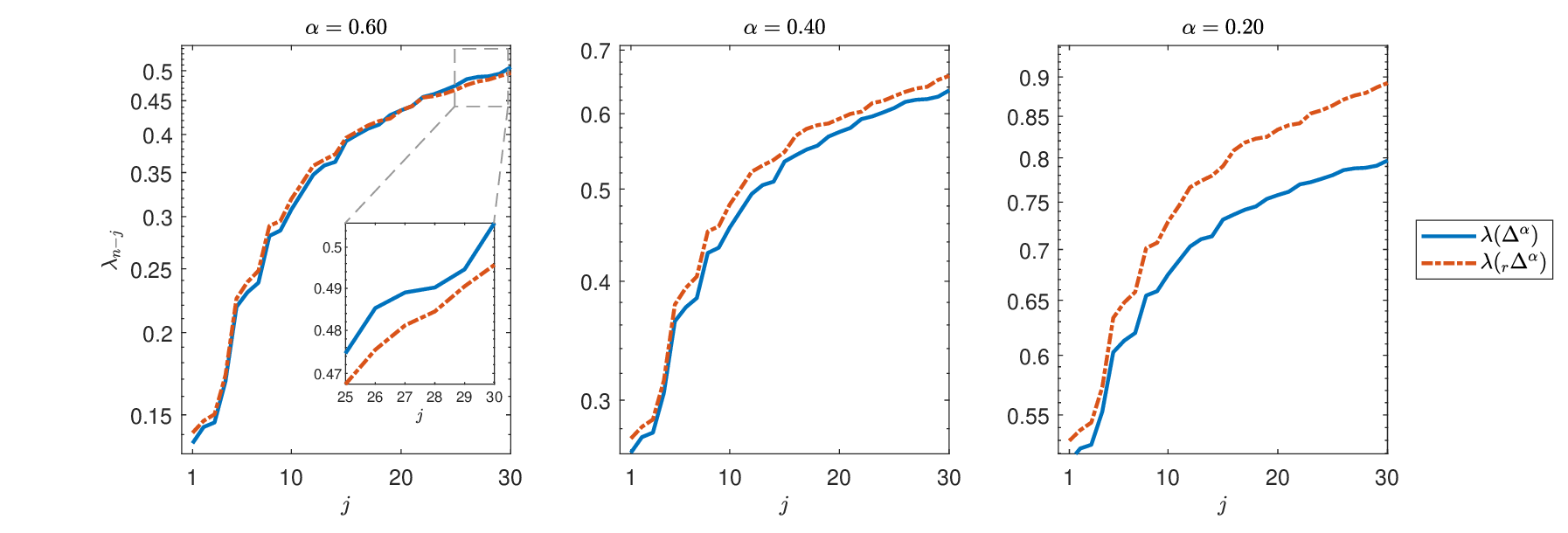}
    \caption { First $30$ nonzero eigenvalues of $\fL$ and $\rfL$ related to the \texttt{YaleA} network for $\alpha=0.60,0.40, 0.20$.}
    \label{fig:spectra_that_switch_behavior}
\end{figure}

\subsection{Non-local diffusion} \label{subsec:experiments_non_local_diffusion}

In the following, we analyze the rate at which the solution of the heat equation
\begin{equation}\label{eq:heat_equation_all_operators}
\begin{cases}
    \diff{\mathbf{x}(t)}{t} +L\mathbf{x}=0, & t>0, \\
    \mathbf{x}(0) = \mathbf{x}_0, 
\end{cases} 
\end{equation}
for $L \in\{\Delta, \fL, \rfL$\}, and $\alpha \in \{0.9, 0.5, 0.1\}$ 
converges to the equilibrium $\overline{\xvec}_{0}$ given the (randomly generated) initial vector $\xvec_0\in\{0,1\}^{n}$. 
To do that, we examine the time needed to reach the threshold value of $10^{-10}$ for
\begin{equation*}
    \delta(t) = \frac {\| \exp(-tL)\mathbf{x}_0 - \overline{\xvec}_0 \| }{\|\xvec_0\|},
\end{equation*}
that is, the 2-norm of the \emph{disagreement vector} relative to the norm of $\xvec_0$. The results for the \texttt{YaleA} and \texttt{ghc1} networks are shown in Figure \ref{fig:consensus}. 

In the case of the \texttt{YaleA} network (Figure \ref{subfig:consensus_YaleA}), we see that the solution of the heat equations associated with the fractional Laplacian and its regularization converges faster if compared to the one associated with the standard Laplacian.  
This does not happen for the \texttt{ghc1} network (Figure \ref{subfig:consensus_ghc1}), where the fact that $\lambda_{n-1}(\Delta)>1$ causes the solution of the equation associated with the fractional Laplacian to converge slower than solution of the standard heat equation.
Note that, unlike $\fL$, the regularized fractional Laplacian maintains its superdiffusive behavior independently of the value of the algebraic connectivity of the graph. The results given in Figure~\ref{fig:consensus} are also summarized in Table~\ref{tab:time_to_hit_threshold}.

\begin{figure}[th]
    \centering
        \begin{subfigure}{\linewidth}
        \centering
        \includegraphics[width=\linewidth]{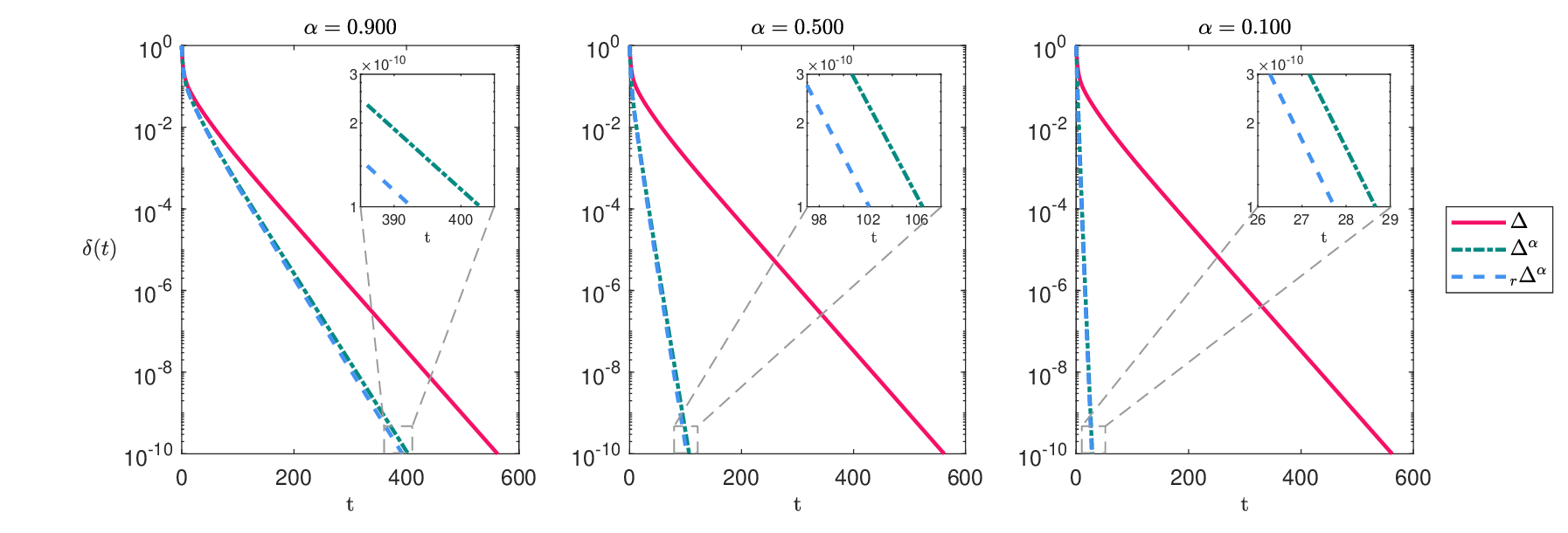}
        \caption{\texttt{YaleA} network}
        \label{subfig:consensus_YaleA}
        \end{subfigure}
        \begin{subfigure}{\linewidth}
        \centering
        \includegraphics[width=\linewidth]{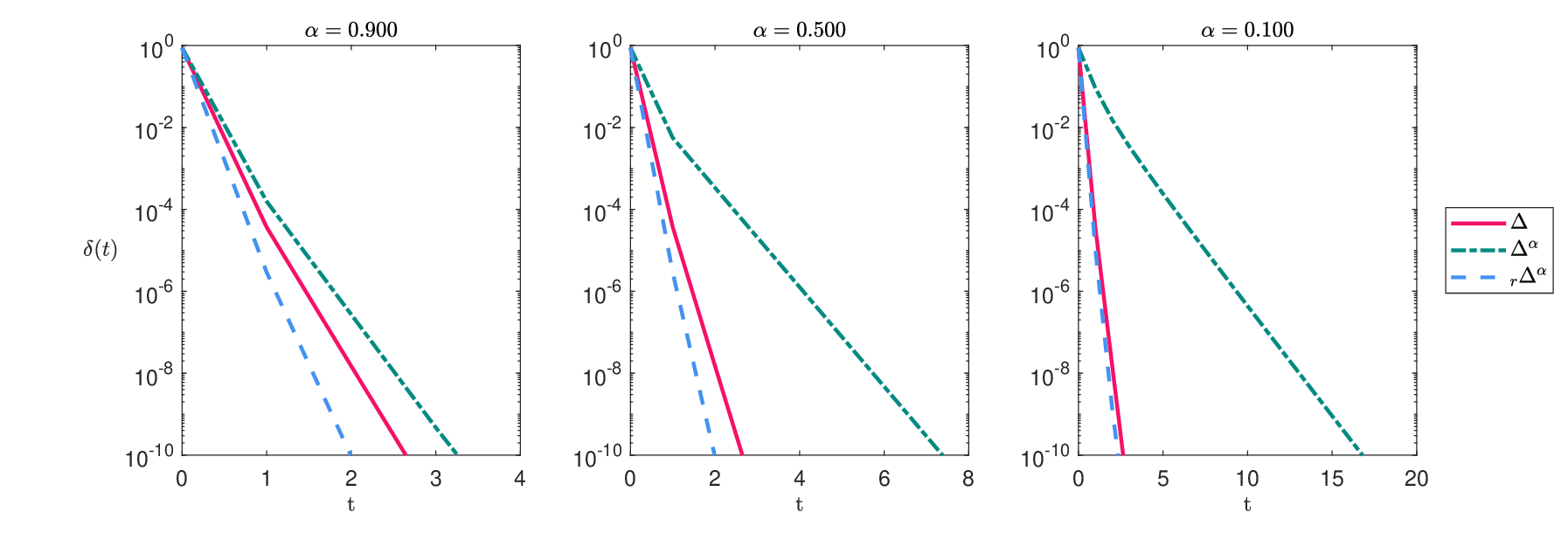}
        \caption{\texttt{ghc1} network}
        \label{subfig:consensus_ghc1}
        \end{subfigure}
    \caption{Time steps needed for $\delta(t)$, i.e., the relative 2-norm of the disagreement vector, to reach the selected threshold of $10^{-10}$ for the \texttt{YaleA} and the \texttt{ghc1} networks. The three lines in each plot represent the speed at which the solution of equation \eqref{eq:heat_equation_all_operators} with $L=\Delta,\fL,\rfL$ converges to the consensus equilibrium. Here we let $\alpha$ vary in $\{0.9,0.5,0.1\}$} 
    \label{fig:consensus}
\end{figure}

\begin{table}[t]
\caption{Time steps required for $\delta(t)$ to hit the threshold of $10^{-10}$ for the \texttt{YaleA} and the \texttt{ghc1} networks. Each column is related to the solution of equation \eqref{eq:heat_equation_all_operators} with $L=\Delta,\fL,\rfL$}
    \centering
    \begin{tabular}{lccccccc}
    \toprule
    \multirow{2}*{Network} & \multirow{2}*{$\Delta$} & \multicolumn{3}{c}{$\fL$} & \multicolumn{3}{c}{$\rfL$} \\
     &    & $\alpha=0.9$  & $\alpha=0.5$  & $\alpha=0.1$  & $\alpha=0.9$ &  $\alpha=0.5$  & $\alpha=0.1$  \\
	\midrule
     \texttt{YaleA} & 562   & 403 & 107 & 29 & 393 & 103 & 28  \\
     \texttt{ghc1} & 3 & 4 & 8& 17 & 2 & 2 & 3  \\
     \bottomrule
    \end{tabular}
    \label{tab:time_to_hit_threshold}
\end{table}

Now let us focus on Figure \ref{fig:consensus_with_PL}, where we extend the above experiment by including a comparison with the path Laplacian \cite{estrada2012path_laplacian} and its generalization for weighted graphs given in \cite{Bianchi2022}. In particular, we study the behavior of $\delta(t)$ for $L=\Delta,\fL,\rfL,\Delta_{s}$, where $\alpha$ is fixed at $0.5$, and $\Delta_{s}$ indicates the path Laplacian with various choices of the parameter $s>0$. 

For the unweighted \texttt{karate} network (Figure \ref{subfig:karate}), we have that the path Laplacian is always superdiffusive compared to the standard graph Laplacian. These results are in line with the ones in \cite{Estrada_2021_fractional_vs_path_laplacians}, where it is stated that, for an unweighted graph, $\lambda_{n-1}(\Delta_{s})<\lambda_{n-1}(\Delta)$ for all $s>0$.
In the case of the weighted \texttt{ghc2} network (Figure \ref{subfig:consensus_ghc2}), we see that the solution of the heat equation associated with the path Laplacian loses its superdiffusive behavior for certain choices of the parameter since, for example, $\lambda_{n-1}(\Delta_{1.25})=0.56<\lambda_{n-1}(\Delta)=1.52$. The results given in Figure~\ref{fig:consensus_with_PL} are also summarized in Table~\ref{tab:time_to_hit_threshold_with_PL}.

\begin{table}[t]
\caption{Time steps required for $\delta(t)$ to hit the threshold of $10^{-10}$ for the \texttt{karate} network and the \texttt{ghc2} network. Each column is related to the solution of equation \eqref{eq:heat_equation_all_operators} with $L=\Delta,\fL,\rfL,\Delta_s$} 
    \centering
    \begin{tabular}{lccccccccc}
    \toprule
    \multirow{2}*{Network} & \multirow{2}*{$\Delta$} & $\fL$ & $\rfL$  & \multicolumn{6}{c}{$\Delta_{s}$} \\
         &   & $\alpha=0.5$ & $\alpha=0.5$ & $s=6$  & $s=4$  & $s=2$  & $s=1.25$ &  $s=0.85$  & $s=0.5$  \\
    \midrule
     \texttt{karate}  & 46 & 32 & 21 & 37 & 21 & 6 & -- & -- & --\\
     \texttt{ghc2} & 14 & 17 & 8 & -- & -- & -- & 37 &  8  & 2  \\
     \bottomrule
    \end{tabular}
    \label{tab:time_to_hit_threshold_with_PL}
\end{table}

Overall, the speed at which $\delta(t)$ reaches the selected threshold in case of the path Laplacian depends on $s$ and can be above, on par, or below the one of the fractional alternatives. On the contrary, the regularized fractional operator is always superdiffusive for any choice of the fractional parameter, and independently of whether the tested graph is weighted or unweighted.

\begin{figure}[th]
        \begin{subfigure}{\linewidth}
        \centering
        \includegraphics[width=\linewidth]{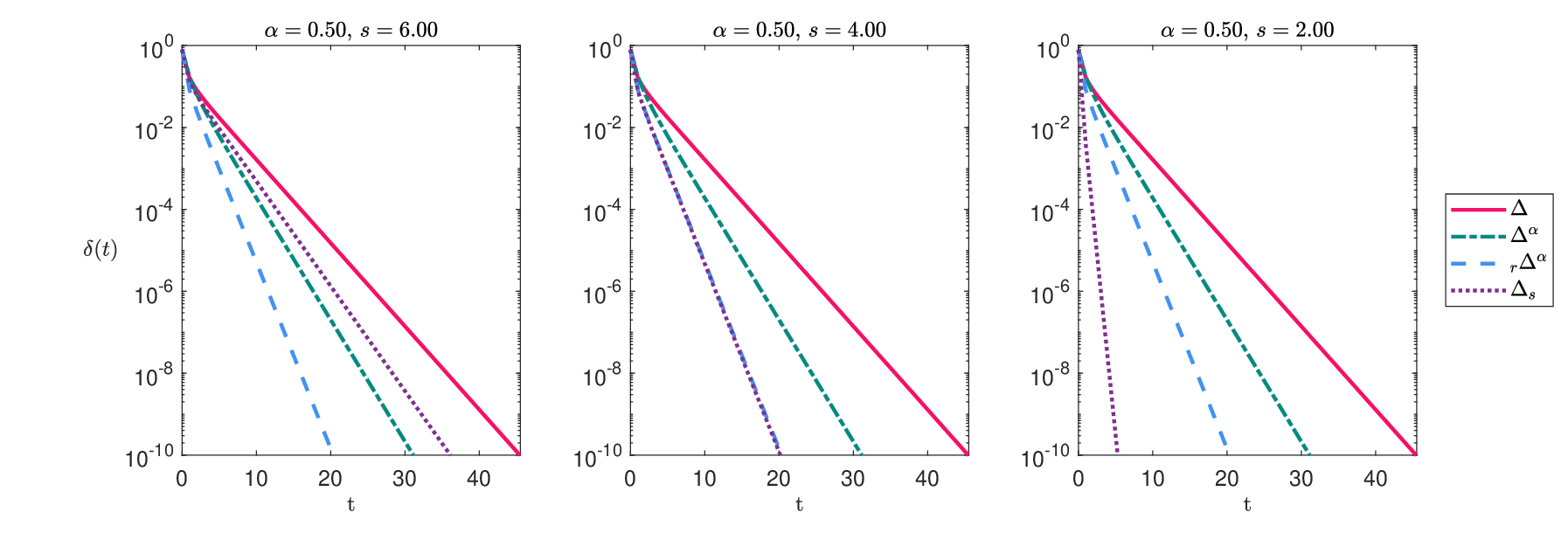}
        \caption{\texttt{karate} network}
        \label{subfig:karate}
        \end{subfigure}
        \begin{subfigure}{\linewidth}
        \centering
        \includegraphics[width=\linewidth]{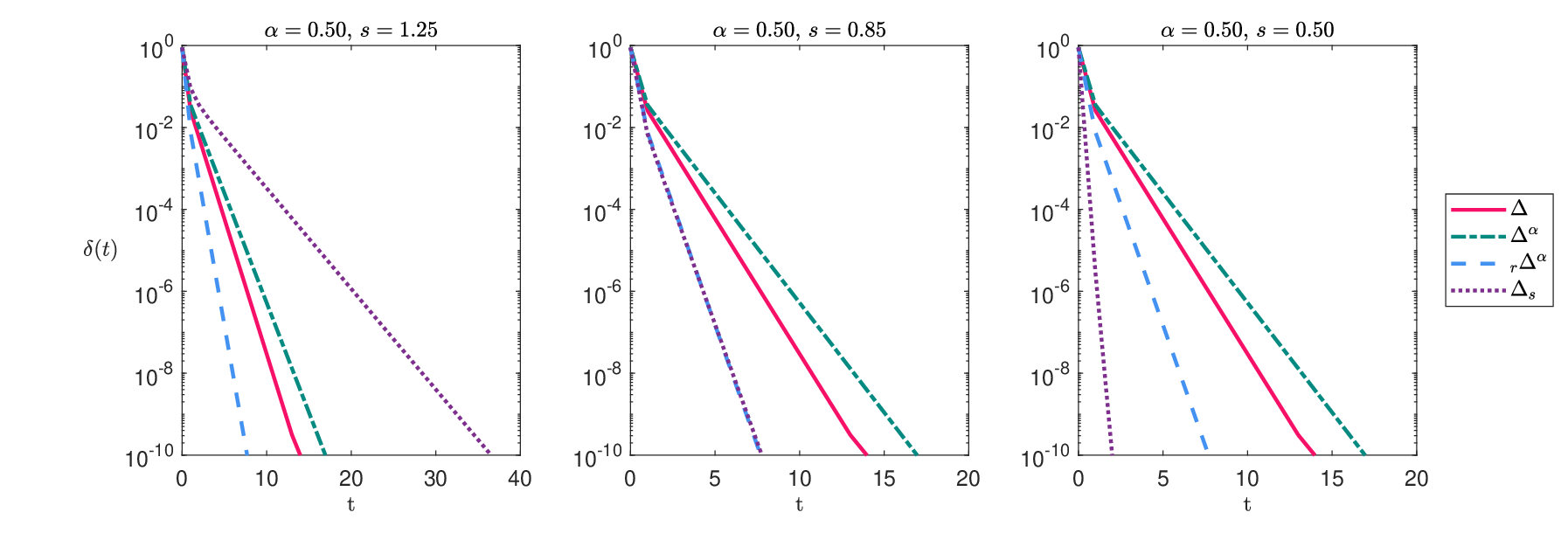}
        \caption{\texttt{ghc2} network}
        \label{subfig:consensus_ghc2}
        \end{subfigure}
    \caption{Time steps needed for $\delta(t)$, i.e., the relative 2-norm of the disagreement vector, to reach the selected threshold of $10^{-10}$ for the unweighted \texttt{karate} network and the weighted \texttt{ghc2} network. The three lines in each plot represent the speed at which the solution of equation \eqref{eq:heat_equation_all_operators} with $L=\Delta,\fL,\rfL,\Delta_{s}$ converges to the consensus equilibrium. Here we fix $\alpha=0.5$ and we let $s$ vary in $\{6,4,2,1.25,0.85,0.5\}$ } 
    \label{fig:consensus_with_PL}
\end{figure}

\subsection{Timings for synthetic and real-world networks}

In this series of experiments, we test the two computational methods discussed in Section \ref{subsec:building_regularized} for computing 
$\rfL$. In particular, we fix $\alpha=0.5$ and compare the execution times of the following algorithms:

\begin{enumerate}
    \item the naive substitution that computes $\rfW$ using two loops and then Equation~\eqref{eq:reg_laplacian_defin};
    \item the newly introduced approach that uses  Equation~\eqref{eq:reg_laplacian_formula}.
\end{enumerate}

In Figure \ref{fig:histogram_computational_times}, we report the absolute CPU timings for some real-world networks in the dataset. In Figure \ref{fig:timings} we show the computational times for synthetic \emph{scale-free} networks with increasing sizes that were generated using the function \texttt{pref(n,4)} taken from the {CONTEST} Toolbox~\cite{CONTEST}. 

From both Figures  \ref{fig:histogram_computational_times} and \ref{subfig:timings_fL_excluded}, we see that the new method based on \eqref{eq:reg_laplacian_formula} constantly outperforms the naive substitution algorithm, especially in the case of the largest tested networks, when we achieve a speedup of about $100\times$. These results are largely expected, since our algorithm exploits the sparsity of the graph (all the tested graphs are, in fact, sparse) while the naive substitution algorithm does not. In addition, from Figure~\ref{fig:histogram_computational_times} we see that in case of the real-world networks in Table~\ref{tab:real_networks_info} the measured CPU times for computing $\rfL$ are about the same of computing $\fL$. This confirms what predicted by Proposition \ref{prop:comp_cost_laplacian}. Similar results are obtained also for the synthetic networks in Figure~\ref{subfig:timings_fL_included}, where the two reported timing plots cannot be distinguished. 

\begin{figure}[th]
    \centering
    \includegraphics[width=0.85\linewidth]{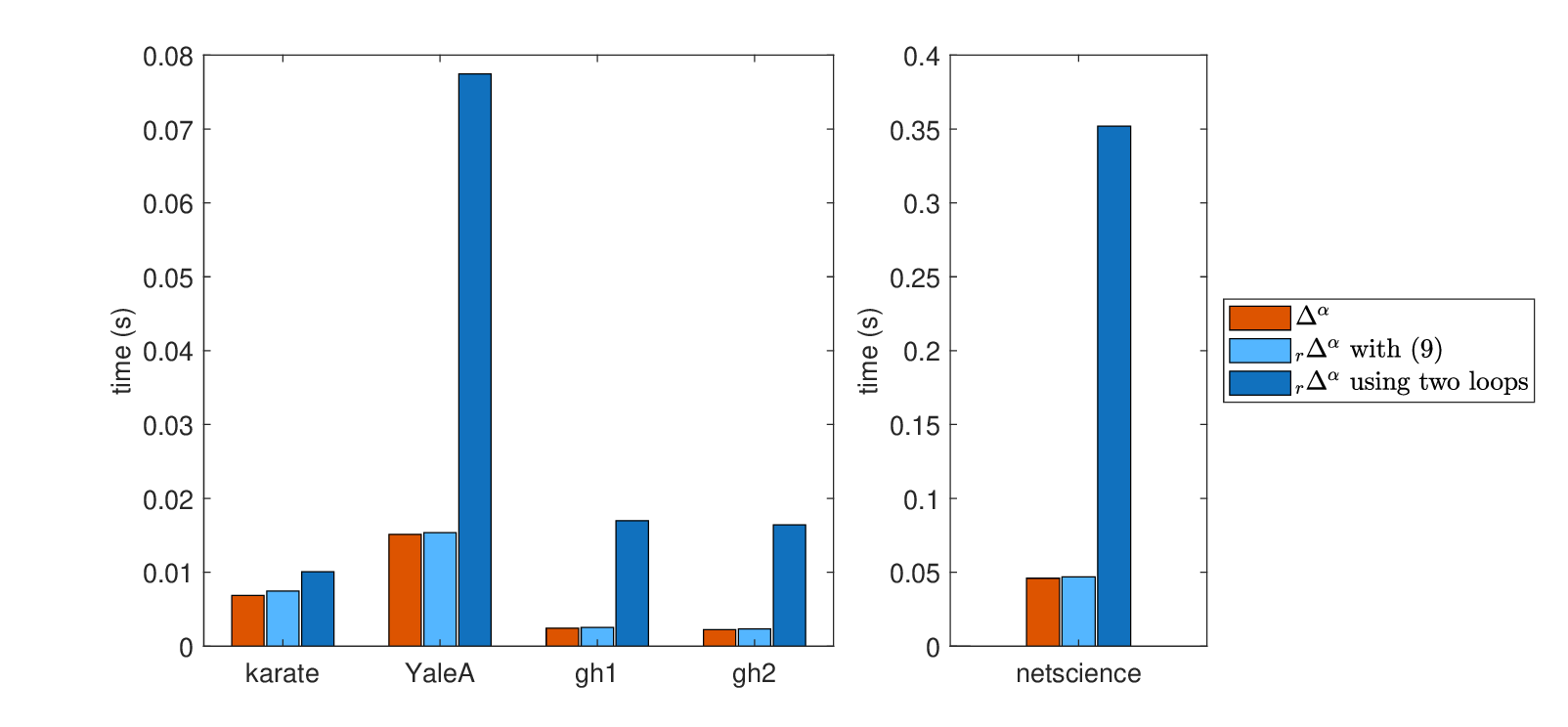}
    \caption{CPU times (in seconds) for computing $\rfL$ with (\ref{eq:reg_laplacian_formula}) (second bar) and the naive substitution algorithm which requires two loops (third bar), for the real-world networks described in Table \ref{tab:real_networks_info}. The timings include the computation of $\fL$ (first bar). 
    All values are averaged over 20 runs with the same parameters. The CPU times for the \texttt{netscience} network are displayed separately because the time scale is much larger than in the other cases} 
    \label{fig:histogram_computational_times}
\end{figure}

\begin{figure}[ht]
\begin{subfigure}{0.47\linewidth}
    \centering
    \includegraphics[width=\linewidth]{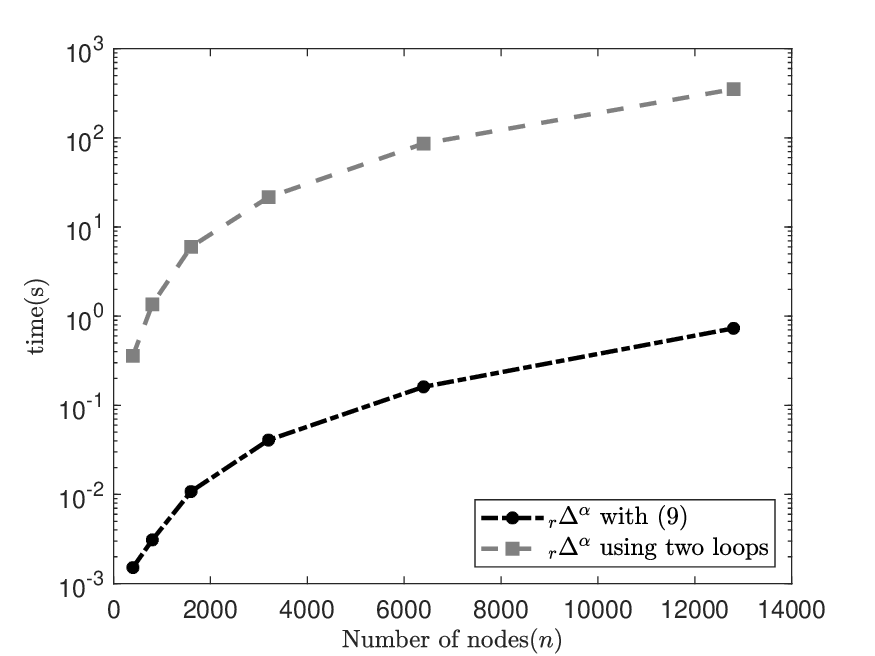} 
    \caption{Timings for synthetic networks (spectral decomposition excluded)}
    \label{subfig:timings_fL_excluded}
\end{subfigure}
\hfill
\begin{subfigure}{0.47\linewidth}
    \centering
    \includegraphics[width=\linewidth]{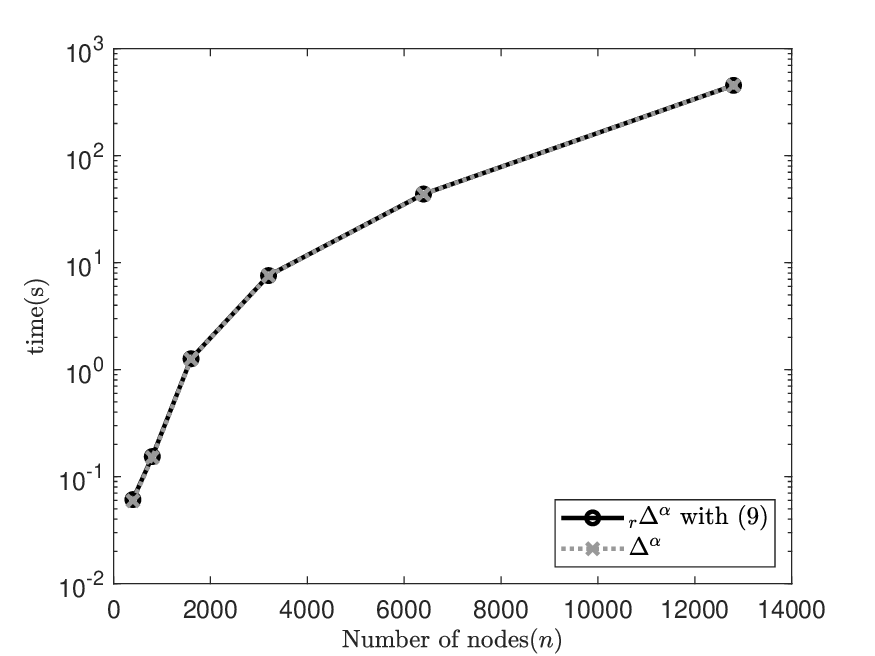}  
    \caption{Timings for synthetic networks (spectral decomposition included) }
    \label{subfig:timings_fL_included}
\end{subfigure}
    \caption{ (A) CPU times in seconds, excluding the spectral decomposition cost, for computing $\rfL$ with \eqref{eq:reg_laplacian_formula} and the naive substitution algorithm that requires two loops. (B) CPU times in seconds, including the spectral decomposition cost, for computing $\fL$ and $\rfL$  with \eqref{eq:reg_laplacian_formula}. In both (A)-(B) we consider synthetic networks of increasing sizes and times are averaged over $20$ runs with the same parameters}
\label{fig:timings}
\end{figure}

\section{Conclusions and open questions} 
\label{sec:conclusions}

In this paper, we analyzed the spectral and computational properties of the regularized fractional Laplacian $\rfL$ as a non-local operator for undirected graphs. In particular, we prove that it always yields improved diffusion performance independently of whether the underlying graph is weighted or unweighted—a property not generally satisfied by the fractional Laplacian or other non-local Laplacian variants. We also described an Boolean--Hadamard computational method to compute $\rfL$ which has the same theoretical cost of computing the fractional Laplacian and is cheaper than that of the naive substitution algorithm which can require up to $n^2$ comparisons.

Numerical tests on both synthetic and real-world networks confirmed the theoretical results and showed the regularized fractional Laplacian is the only operator ensuring superdiffusivity in all tested configurations when compared to the fractional and the path Laplacians. 

Open questions are whether this work can be extended to directed graphs \cite{Benzi_2020_fractional_laplacian}, and whether some of these results generalize to the \emph{normalized graph Laplacian} and its respective non-local variants \cite{Riascos_introduction_of_fractional_laplacian}.

\paragraph*{Acknowledgments}
 The authors thank the anonymous referees for their suggestions which improved the quality of the manuscript. The authors also gratefully acknowledge Emilia Cozzolino for helpful discussions and for pointing them towards the datasets discussed in \cite{braingraph}. Both authors are members of INdAM-GNCS.

\paragraph{Funding}
This work has been partially financed by the INdAM-GNCS Project CUP E53C25002010001. MM acknowledges the MUR Excellence Department Project MatMod@TOV awarded to the Department of Mathematics, University of Rome Tor Vergata, CUP E83C23000330006.

\bibliographystyle{abbrv}
\bibliography{ref}

\end{document}